\documentclass[12pt]{article} 
\usepackage{setspace}
\usepackage[utf8]{inputenc} 

\usepackage{subfigure}
\usepackage[margin=1in]{geometry}
\usepackage{graphicx} 
\usepackage{epstopdf}

\usepackage{booktabs} 
\usepackage{array} 
\usepackage{paralist} 
\usepackage{verbatim} 
\usepackage{subfig} 

\usepackage{fancyhdr} 
\usepackage{sectsty}
\allsectionsfont{\sffamily\mdseries\upshape} 

\usepackage[nottoc,notlof,notlot]{tocbibind} 
\usepackage[titles,subfigure]{tocloft} 

\usepackage{amsmath, amsthm, amssymb}

\usepackage{multirow}
\usepackage{algorithm}
\usepackage{algorithmic}

\usepackage{amsthm}
\newtheorem{thm}{Theorem}[section]

\newtheorem{cor}{Corollary}[section]
\newtheorem{dfn}{Definition}[section]
\newtheorem{prop}{Proposition}[section]

\usepackage{tikz, pgfplots}
\pgfplotsset{compat=1.6} 
\usetikzlibrary{calc,positioning,shapes,fit,arrows,shadows}
\tikzstyle{line} = [ draw, -latex']  
\usetikzlibrary{shapes,arrows}
\usepackage{url}
\usepackage{color}
\usepackage{rotating}
\usepackage{pdflscape}
\allowdisplaybreaks 
\graphicspath{{figures/}} 

\DeclareMathOperator{\atan2}{atan2}
\DeclareMathOperator{\argmin}{argmin}

\title{Matrix Minor Reformulation and SOCP-based Spatial Branch-and-Cut Method for the AC Optimal Power Flow Problem} 
\author{Burak Kocuk, Santanu S. Dey, X. Andy Sun} 

\begin{document}
\maketitle

\begin{abstract}
Alternating current optimal power flow (AC OPF) is one of the most fundamental optimization problems in electrical power systems. It can be formulated as a semidefinite program (SDP) with rank constraints. Solving AC OPF, that is, obtaining near optimal primal solutions as well as high quality dual bounds for this non-convex program, presents a major computational challenge to today's power industry for the real-time operation of large-scale power grids. In this paper, we propose a new technique for reformulation of the rank constraints using both principal and non-principal $2$-by-$2$ minors of the involved Hermitian matrix variable and characterize all such minors into three types. We show the equivalence of these minor constraints to the physical constraints of voltage angle differences summing to zero over three- and four-cycles in the power network. We study second-order conic programming (SOCP) relaxations of this minor reformulation and propose strong cutting planes, convex envelopes, and bound tightening techniques to strengthen the resulting SOCP relaxations. We then propose an SOCP-based spatial branch-and-cut method to obtain the global optimum of AC OPF. Extensive computational experiments show that the proposed algorithm significantly outperforms the state-of-the-art SDP-based OPF solver and on a simple personal computer is able to obtain on average a $0.71\%$ optimality gap in no more than $720$ seconds for the most challenging power system instances in the literature.   

\end{abstract}

\section{Introduction}
\subsection{Approaches to solving alternating current optimal power flow problem}
Alternating current optimal power flow (AC OPF) is one of the most fundamental optimization problems in electrical power system operation. Since its introduction in 1962 \cite{Carpentier}, developing efficient algorithms for solving AC OPF has remained an active field of research. The main challenge in solving AC OPF in order to obtain a near-optimal primal solution and a high quality dual bound lies in the non-convex nature of the AC OPF model, the very large scale of the resulting optimization problem for real-world power grids, and the very short amount of solution time required by the real-time dispatch of the generation and demand resources in the power grid. 

Literature on the AC OPF problem can be roughly categorized into three groups: local methods, relaxation approaches, and global optimization methods (e.g. see recent surveys in \cite{Momoh11999, Momoh21999, Frank12012, Frank22012, FERC2012}). In the local methods, the aim is to find a local optimal solution or a stationary point of the OPF problem. Typically, these methods are based on the Newton-Raphson algorithm or the interior point methods (IPM) \cite{Wu1994, Torres1998, Jabr2002, Zimmerman2007, Matpower}, and are effective in locating good feasible solutions only under normal operating conditions (e.g. when the power system is not stressed by high level of load and/or tight transmission constraints).  
However, since the OPF problem is non-convex, such local methods may get stuck at local optimal or stationary solutions \cite{Bukhsh,kocuk2014}. Also, local methods are sensitive to the initial point and may fail to produce solutions when a warm starting point is not available \cite{kocuk2014}.

%
%
%
%
%
%

Relaxation approaches for solving AC OPF are mainly based on the semidefinite programming (SDP) relaxation techniques  \cite{Bai2008, Bai2009, Lavaei12}. It recently become a popular approach as researchers found that the SDP relaxations of the standard IEEE test instances up to 300 buses all obtained globally optimal solutions \cite{Lavaei12}. 
However, it was soon realized that the requirements to guarantee exactness of the SDP relaxation are typically quite restrictive (e.g. resistive networks with no reactive loads \cite{Lavaei12}, networks with a large number of virtual phase shifters  \cite{sojoudi2012physics}, radial networks with no generation lower bounds \cite{Zhang12, bose2011optimal, bose2012quadratically}). Furthermore, when the relaxation is not exact, it is a challenge to recover a feasible primal solution. However, since the dual bounds obtained by SDP relaxations are typically strong, this approach proves to be a valuable tool for evaluating the quality of the primal feasible solutions obtained by other methods, e.g. a local method. Last but not the least, solving SDP is not very scalable with the existing algorithms. Methods exploiting the sparsity pattern of the underlying graph have been proposed for general polynomial optimization problems in in  \cite{fukuda2001} and \cite{nakata2003}, and later applied to the OPF problem in \cite{jabr2012, molzahn2013, madani2014, molzahn2014, madani2015}. Despite the considerable efforts, solving large-scale AC OPF with SDP relaxation is still a computational challenge for real-world power systems.

The scalability issue of SDP-based relaxations motivates researchers (including the present paper) to pursue simpler convex relaxations, such as second-order cone programming (SOCP) and linear programming (LP) relaxations. Basic SOCP relaxation is first applied to the OPF problem in \cite{Jabr06}. Various strengthening techniques using valid inequalities, convex envelopes, and cutting planes are proposed to strengthen the SOCP relaxation, and the resulting strong SOCP relaxations prove to be not dominated by SDP relaxations and can be solved much faster \cite{kocuk2014, kocuk2015, hijazi2013, coffrin2015}. LP approximation and relaxation techniques are also applied to the OPF problem in \cite{coffrin2014, bienstock2014}. Global methods try to combine the above methods. A particular work in this area is \cite{Phan}, where a local solver (IPOPT) is used to find feasible primal solutions and a Lagrangian relaxation algorithm is used to obtain dual bounds in a spatial branch-and-bound framework. However, the proposed approach in \cite{Phan} is only tested on standard IEEE instances, which turn out to be simple instances and can all be solvable in the root node. A recent work \cite{Chen2015Report} also proposed a branch-and-cut algorithm for solving complex quadratic constrained quadratic programs with applications to AC OPF.

\subsection{Our contributions}
In this paper, we aim to solve the AC OPF problem to global optimum. We first propose a new systematic approach to dealing with the rank constraint of a complex Hermitian matrix in a nonconvex QCQP in lifted space. 
In particular, we use the fact that the enforcing the rank of a Hermitian matrix being equal to one is equivalent to imposing the condition that \emph{all} the $2\times 2$ minors of the matrix are zero. 
These minor constraints can be classified into three types. Each type can be further reformulated as a set of quadratic or bilinear constraints involving a small set of variables. This reformulation applies to any non-convex quadratic problem with Hermitian matrices and it lays the foundation for a systematic approach to relaxing and/or convexifying the rank constraints. Interestingly, these three types of minors also have clear interpretations. In particular, they correspond to edges, 3-cycles (i.e. cycles of three nodes), and 4-cycles (i.e. cycles of 4 nodes) of the underlying network. In the context of AC OPF, these cycle constraints have a clear physical meaning, e.g. voltage angle differences of edges in a cycle sum to zero. In \cite{Chen2015IEEE, Chen2015Report, Hijazi2015SDP}, only \textit{principal} $2\times2$ minors, corresponding to Type 1 minors in our language, are considered to improve the SDP relaxation. By considering the other two types of non-principal minors, we are able to obtain stronger relaxations than existing ones.

Based on this minor reformulation, we propose convexification techniques to exploit the special structures of cycles in a graph, and design and implement a complete spatial branch-and-cut algorithm based on second-order conic programming (SOCP) for obtaining global or near global optimal solutions of AC OPF. Extensive computational experiments show that the proposed algorithm significantly outperforms the state-of-the-art SDP-based OPF solver and 
 on a simple personal computer is able to obtain on average a $0.71\%$ optimality gap in no more than $720$ seconds for the most challenging power system instances in the literature.   

The rest of the paper is organized as follows: In Section \ref{sec:opfintro}, we formally define the AC OPF problem. In Section \ref{sec:rank constrained}, we give the new reformulation of a complex SDP problem with rank constraint as a real quadratic optimization problem with minor constraints. Then, we propose several outer-approximation schemes to incorporate the minor conditions in a convex relaxation of the AC OPF problem. Section \ref{sec:alg socp} presents our SOCP based spatial branch-and-cut algorithm with a particular emphasis on the root node relaxation. Our extensive computational experiments on challenging NESTA instances \cite{nesta} are summarized in Section \ref{sec: comp exper}. Finally, Section \ref{sec: concl} concludes the paper with some further remarks and possible future research directions.

\section{AC Optimal Power Flow Problem}\label{sec:opfintro}
In this section, we present the mathematical optimization formulation of the AC OPF problem using the so-called rectangular formulation. 
Consider a power network $\mathcal{N} = (\mathcal{B},\mathcal{L})$, where $\mathcal{B}$ denotes the set of nodes (or buses),  and $\mathcal{L}$ denotes the set of edges (or transmission lines). Electric power generators are attached to a subset of buses, which is denoted as $\mathcal{G}\subseteq \mathcal{B}$. We assume that there is electric demand (or load) at every bus, some of which might potentially be zero. The aim of the AC OPF problem is to determine generation levels of generators to satisfy demand with the minimum total generation cost in such a way that various physical constraints and operational constraints are satisfied.

Let $Y \in \mathbb{C}^{|\mathcal{B}| \times |\mathcal{B}|}$ denote the nodal admittance matrix, which has components $Y_{ij}=G_{ij} + \mathrm{i}B_{ij}$ for each line $(i,j)\in\mathcal{L}$, and $G_{ii}=g_{ii}-\sum_{j\ne i} G_{ij}, B_{ii}=b_{ii}-\sum_{j\ne i} B_{ij}$, where $g_{ii}$ and 
$b_{ii}$ are respectively  the shunt conductance and susceptance at bus $i\in\mathcal{B}$. Here, $\mathrm{i}=\sqrt{-1}$. 
Let $p_i^d$ and  $q_i^d$ be the real and reactive load at bus $i$ given as part of the data. 
We also define decision variables $p_i^g$ and  $q_i^g$ to represent the real and reactive power output of the generator at bus $i$. Another set of decision variables is the complex voltage (also called voltage phasor) $V_i$ at bus $i$, which can be expressed in the rectangular form as $V_i = e_i+\mathrm{i} f_i$.
It is sometimes convenient to represent $V_i$ in the polar form as $V_i = |V_i|(\cos\theta_i+\mathrm{i}\sin\theta_i)$, where $|V_i|^2=e_i^2 + f_i^2$ is the  magnitude and $\theta_i$ is the angle of the complex voltage. 

Finally, the AC OPF problem in the rectangular formulation is given below~\cite{Carpentier}.
\vspace{-5mm}
\begin{subequations}\label{rect}
\begin{align}
\min  &\hspace{0.25em}  \sum_{i \in \mathcal{G}} C_i(p_i^g)  \label{obj} \\
  \mathrm{s.t.}   &\hspace{0.25em} p_i^g-p_i^d = G_{ii}(e_i^2+f_i^2) + \sum_{j \in \delta(i)}[ G_{ij}(e_ie_j+f_if_j) -B_{ij}(e_if_j-e_jf_i)]   & i& \in \mathcal{B} \label{activeAtBus} \\
  & \hspace{0.25em} q_i^g-q_i^d = -B_{ii}(e_i^2+f_i^2) + \sum_{j \in \delta(i)}[ -B_{ij}(e_ie_j+f_if_j) -G_{ij}(e_if_j-e_jf_i)]  & i& \in \mathcal{B} \label{reactiveAtBus} \\
  & \hspace{0.25em} \underline V_i^2 \le e_i^2+f_i^2 \le \overline V_i^2    & i& \in \mathcal{B} \label{voltageAtBus} \\
  & \hspace{0.25em}  p_i^{\text{min}}  \le p_i^g \le p_i^{\text{max}}     & i& \in \mathcal{B} \label{activeAtGenerator} \\
  & \hspace{0.25em} q_i^{\text{min}}  \le q_i^g \le q_i^{\text{max}}     & i& \in \mathcal{B}. \label{reactiveAtGenerator}\\
  & \hspace{0.5em}  [-G_{ij}(e_i^2+f_i^2) + G_{ij}(e_ie_j+f_if_j) -B_{ij}(e_if_j-e_jf_i)]^2 \nonumber  \\
 &+ [B_{ij}(e_i^2+f_i^2) -  B_{ij}(e_ie_j+f_if_j) - G_{ij}(e_if_j-e_jf_i)]^2   \le (S_{ij}^{\text{max}})^2  &(&i,j) \in \mathcal{L} \label{powerOnArc} 
\end{align}
\end{subequations}

%

In the OPF formulation, the objective function is separable over the generators and $C_i(p_i^g)$ is typically linear or convex quadratic in the real power output $p_i^g$.
Constraints \eqref{activeAtBus} and \eqref{reactiveAtBus} correspond to the Kirchoff's Current Law, i.e. the conservation of active and reactive power flows at each bus, respectively, where $\delta(i)$ denotes the set of neighbors of bus $i$. 
Constraint \eqref{voltageAtBus} limits voltage magnitude at each bus, which is usually normalized against a unit voltage level and is expressed in per unit (p.u.). Therefore, $\underline V_i$ and $\overline V_i$ are both close to  1 p.u. at each bus $i$. This strict requirement is crucial to maintain the stability across the power system. 
Constraints \eqref{activeAtGenerator} and \eqref{reactiveAtGenerator}, respectively, restrict the active and reactive power output of each generator to their physical capability.  Here, we set $ p_i^{\text{min}} =  p_i^{\text{max}}=q_i^{\text{min}} =  q_i^{\text{max}}=0$ for non-generator bus $i \in \mathcal{B} \setminus \mathcal{G}$.
Finally, constraint \eqref{powerOnArc} puts an upper bound on the total electric power on each transmission line $(i,j)$. 

\section{Minor-based Reformulation of AC OPF Problem}
\label{sec:rank constrained}
\subsection{Standard rank formulation}

The AC OPF problem given in \eqref{rect} is a polynomial optimization problem involving quadratic and quartic polynomials, and it can be posed as a quadratically constrained quadratic program (QCQP) by rewriting the constraint \eqref{powerOnArc} using additional variables. For nonconvex QCQPs, a standard strategy is to {lift} the variable into a higher dimensional space \cite{Nesterov2000}. This can be accomplished by defining a matrix variable to replace the quadratic terms in the original variables with linear ones in the lifted matrix variable. In this reformulation, the only nonconvex constraint is a requirement that the rank of the matrix variable is one, which we call the rank-one constraint. The procedure to obtain this lifted formulation for the OPF problem is given below.

First, let us define a Hermitian matrix $X\in\mathbb{C}^{|\mathcal{B}|\times|\mathcal{B}|}$, i.e., $X=X^*$, where $X^*$ is the conjugate transpose of $X$. Consider the following set of constraints:
\begin{subequations} \label{rect sdp}
\begin{align}
& p_i^g-p_i^d = G_{ii}X_{ii} + \sum_{j \in \delta(i)}[ G_{ij} \Re(X_{ij}) +B_{ij} \Im(X_{ij})]   \hspace{4cm} \forall i \in \mathcal{B} \label{rect sdp first}  \\
& q_i^g-q_i^d = -B_{ii}X_{ii} + \sum_{j \in \delta(i)}[ -B_{ij}\Re(X_{ij}) +G_{ij}\Im(X_{ij})]  \hspace{3.4cm} \forall i \in \mathcal{B} \label{rect sdp second} \\
&  \underline V_i^2 \le X_{ii} \le \overline V_i^2   \hspace{10cm} \forall i \in \mathcal{B}\label{rect sdp last} \\
& [-G_{ij}X_{ii} + G_{ij}\Re(X_{ij}) +B_{ij} \Im(X_{ij})]^2 +[B_{ij}X_{ii}-B_{ij}\Re(X_{ij}) +G_{ij} \Im(X_{ij})]^2  \le (S_{ij}^{\text{max}})^2  \; \forall (i,j) \in \mathcal{L}   \label{rect sdp flow}   \\ 
& X \textup{ is Hermitian}   \label{hermitian} \\
& X \succeq  0   \label{psd cons} \\
& \text{rank} (X)= 1,   \label{psd rank} 
\end{align}
\end{subequations}
where $\Re(x)$ and $\Im(x)$ are the real and imaginary parts of the complex number $x$, respectively. Let $V$ denote the vector of voltage phasors with the $i$-th entry $V_i =e_i+\mathrm{i}f_i$ for each bus $i\in\mathcal{B}$.  Then, the AC OPF problem in \eqref{rect} can be equivalently rewritten in the following \emph{rank formulation}:
\begin{equation}\label{rank cons OPF form}
z_{\text{rank}}:=\min \left\{ \sum_{i \in \mathcal{G}} C_i(p_i^g) : \eqref{rect sdp}, \eqref{activeAtGenerator}-\eqref{reactiveAtGenerator} \right\}.
\end{equation}
Let $z_{\text{ACOPF}}$ denote the optimal value of the AC OPF problem in the rectangular formulation \eqref{rect}. Then, clearly $z_{\text{ACOPF}} = z_{\text{rank}}$.
The nonconvexity in the rank formulation \eqref{rank cons OPF form} is captured by the rank constraint \eqref{psd rank}. In order to deal with this challenging constraint, in the following we propose a new reformulation of AC OPF, based on conditions on the minors of matrix~$X$. 

\subsection{New approach: minor-based reformulation}
In this section, we propose a new approach to reformulating the rank-one constraints. First, we recall the definition of an $m\times m$ minor of a matrix \cite{Horn}:
\begin{dfn}
Let $X$ be an $n \times n$ complex matrix. Then, an $m\times m$ minor of the matrix $X$ for $1\le m\le n$ is the determinant of an $m\times m$ submatrix of $X$ obtained by deleting $n-m$ rows and columns from $X$.
\end{dfn}
The following Proposition \ref{rank minor prop} provides the key characterization of the rank-one constraint in terms of $2\times 2$ minors. 
\begin{prop} \label{rank minor prop}
A nonzero, Hermitian matrix $X$ is positive semidefinite and has rank one, i.e. $X \succeq 0$ and $\text{rank} (X)= 1$, if and only if all the $2\times2$ minors of $X$ are zero and the diagonal elements of $X$ are nonnegative.
\end{prop}
\begin{proof}
\noindent($\Rightarrow$) Assume $X \succeq 0$ and $\text{rank} (X)= 1$. Then, there exists a vector $x \in \mathbb{C}^n$ such that $X=xx^*$. Let us consider a $2\times2$ submatrix of $X$ of the form 
\begin{align*}
X_{(i,j),(k,l)}=\begin{bmatrix} X_{ik} & X_{il} \\ X_{jk} & X_{jl} \end{bmatrix}.
\end{align*} Note that we have $X_{ik}=x_i x_k^*$, $X_{il}=x_i x_l^*$, $X_{jk}=x_j x_k^*$, and $X_{jl}=x_j x_l^*$. Hence, $\det X_{(i,j),(k,l)} = 0$, implying that all  the $2\times2$ minors of $X$ are zero. Also, for any diagonal element of the matrix $X$, we have $X_{ii} = x_i x_i^* = |x_i|^2 \ge 0$. Therefore, we conclude that the  diagonal elements of $X$ are nonnegative.

\noindent($\Leftarrow$) Recall that rank$(X)$ is equal to the size of the largest invertible submatrix of $X$ \cite{Horn}. Since  all the $2\times2$ minors of $X$ are zero, none of the $2\times2$ submatrices of $X$ are invertible. This implies that rank$(X)=1$ as $X$ is a nonzero matrix. Finally, since  a rank-one, Hermitian matrix with nonnegative diagonal elements can be written in the form $X=xx^*$  for some vector $x \in \mathbb{C}^n$, we conclude that $X \succeq 0$.
\end{proof}
Based on this property of matrix minors, we can reformulate the AC OPF problem \eqref{rect} as
\begin{equation}\label{minor}
z_{\text{minor}}=\min \left\{ \sum_{i \in \mathcal{G}} C_i(p_i^g) : \eqref{rect sdp first}-\eqref{psd cons}, \eqref{activeAtGenerator}-\eqref{reactiveAtGenerator}, 
\text{ all $2 \times 2$ minors of $X$ are zero} \right\}.
\end{equation}
We call \eqref{minor} the minor-based reformulation, which is an exact reformulation of \eqref{rect}. Therefore, we have $z_{\text{ACOPF}} = z_{\text{rank}} = z_{\text{minor}}$.

\subsection{Discussion on relaxations}
The rank \eqref{rank cons OPF form} and the minor-based \eqref{minor} reformulations lead to systematic ways to construct relaxations of the AC OPF problem \eqref{rect}. For example, the standard SDP relaxation for the OPF problem can be obtained by relaxing the rank constraint \eqref{psd rank}, or equivalently, all the minor constraints in \eqref{minor} as
\begin{equation}\label{OPF sdp relax}
z_{\text{SDP}}=\min \left\{ \sum_{i \in \mathcal{G}} C_i(p_i^g) : \eqref{rect sdp first}-\eqref{psd cons}, \eqref{activeAtGenerator}-\eqref{reactiveAtGenerator} \right\}.
\end{equation}
By construction, $z_{\text{SDP}} \le z_{\text{ACOPF}}$. The standard SDP relaxation has been shown to be exact for the standard IEEE test cases up to 300 buses \cite{Lavaei12, Madani}. However, it is also shown by recent study that, for challenging test cases such as the instances from NESTA archive \cite{nesta}, the duality gap of the standard SDP relaxation \eqref{OPF sdp relax} can be quite large \cite{coffrin2015}. Another challenge of \eqref{OPF sdp relax} is the computational difficulty of solving large-scale SDP problems. Although sparsity exploitation methods help reduce the solution times significantly \cite{molzahn2014}, for truly large-scale power systems, it is necessary to seek computationally less demanding methods.
%

Our approach is to avoid the SDP constraint \eqref{psd cons}, which is computationally expensive to handle, and focus on providing efficient ways to incorporate convex relaxations of the minor constraints in the form of \textit{linear} or \textit{second-order conic} representable constraints. In this approach, some of the proposed constraints will be a convex relaxation of the $2\times2$ minor constraints, discussed in Sections \ref{sec:2by2 minors} and \ref{sec:2by2 minors arctan}; and some will be a weaker version of the semidefiniteness requirement on $X$, discussed in Sections \ref{sec:2by2 minors} and \ref{sec: sdp separate}. The computational cost of the proposed approach can be significantly lower than the standard SDP relaxation. Moreover, the convex relaxation of the minor constraints may provide additional strengthen that is not captured by the standard SDP relaxation. In this way, the proposed approach may be both faster and stronger than the SDP relaxation, which is indeed verified by extensive experiments on the challenging NESTA instances. 

\subsection{Analysis of $2\times2$ Minors}
\label{sec:2by2 minors}

In this section, we analyze all the $2\times2$ minors of the matrix $X$ in detail and characterize them into three different types. Note that for notational purposes, we will define $c_{ij} := \Re(X_{ij})$ and $s_{ij} := -\Im(X_{ij})$ for $i,j \in \mathcal{B}$ in the sequel.
\begin{enumerate}
\item Type 1: Edge Minor. 
Let $i$ and $j$ be distinct elements of the set $\mathcal{B}$. Then, we have
\begin{equation}
 \begin{vmatrix} X_{ii} & X_{ij} \\ X_{ji} & X_{jj} \end{vmatrix} = 0, \label{type 1 minor}
\end{equation}
which is equivalent to 
\begin{equation}
0 = X_{ii}X_{jj} - X_{ij}X_{ji} =  c_{ii}c_{jj}  -  (c_{ij} - \mathrm{i}s_{ij})(c_{ij} + \mathrm{i}s_{ij}) = c_{ii}c_{jj}  - (c_{ij}^2 + s_{ij}^2). \label{type 1 minor cs}
\end{equation}
Note that this relation defines the boundary of the rotated SOCP cone in $\mathbb{R}^4$.

\item Type 2: 3-Cycle Minor.
Let $i$, $j$ and $k$ be distinct elements of the set $\mathcal{B}$, assuming $|\mathcal{B}| \ge 3$. Then, consider the following minor
\begin{equation}
 \begin{vmatrix} X_{ii} & X_{ij} \\ X_{ki} & X_{kj} \end{vmatrix} = 0, \label{type 2 minor}
\end{equation}
which is equivalent to 
\begin{equation}\label{type 2 minor cs}
\begin{split}
0 =& X_{ii}X_{kj} - X_{ij}X_{ki} =  c_{ii}(c_{kj} - \mathrm{i}s_{kj})  -  (c_{ij} - \mathrm{i}s_{ij})(c_{ki} - \mathrm{i}s_{ki}) \\
=& (c_{ii}c_{kj} - c_{ij}c_{ki}  + s_{ij}s_{ki}) -  \mathrm{i}(c_{ii}s_{kj} - s_{ij}c_{ki}  - c_{ij}s_{ki}) .
\end{split}
\end{equation}
Note that this relation defines two bilinear equations in $\mathbb{R}^7$.

\item Type 3: 4-Cycle Minor.
Let $i$, $j$, $k$  and $l$ be distinct elements of the set $\mathcal{B}$, assuming $|\mathcal{B}| \ge 4$. Then, consider the following minor
\begin{equation}
 \begin{vmatrix} X_{ij} & X_{ik} \\ X_{lj} & X_{lk} \end{vmatrix} = 0, \label{type 3 minor}
\end{equation}
which is equivalent to 
\begin{equation}\label{type 3 minor cs}
\begin{split}
0 =& X_{ij}X_{lk} - X_{ij}X_{lj} = (c_{ij} - \mathrm{i}s_{ij})(c_{lk} - \mathrm{i}s_{lk})  -  (c_{ij} - \mathrm{i}s_{ij})(c_{lj} - \mathrm{i}s_{lj}) \\
=& (c_{ij}c_{lk} - s_{ij}s_{lk}  - c_{lj}c_{ik}  + s_{lj}s_{ik}) -  \mathrm{i}(s_{ij}c_{lk} - c_{ij}s_{lk}  - s_{lj}c_{ik}  - c_{lj}s_{ik}) .
\end{split}
\end{equation}
Note that this relation defines two bilinear equations in $\mathbb{R}^8$.
\end{enumerate}

%
%
%
%
%
%
%
%

We analyze these minors and their relaxations in detail below. At this point, we would like to contrast our approach with the existing literature. In \cite{Chen2015IEEE, Chen2015Report, Hijazi2015SDP}, only \textit{principal} $2\times2$ minors, corresponding to Type 1 minors in our language, are considered to improve the SDP relaxation. By considering Type 2 and 3 minors, we can potentially obtain stronger relaxations.

\subsubsection{Type 1: Edge Minors}


Type 1 minors are principal minors of the matrix $X$. The straightforward convex relaxation of \eqref{type 1 minor cs} leads to an SOCP relaxation of the standard SDP relaxation of AC OPF and has been extensively used in the literature \cite{Jabr06, Jabr08, Chen2015IEEE, Chen2015Report, Hijazi2015SDP}. Our goal is to go beyond the simple SOCP relaxation and to obtain the convex hull description of the set defined by a Type 1 minor in \eqref{type 1 minor} over the hypercube  $\mathcal{D}_{ij}:=[\underline c_{ii}, \overline c_{ii}]\times[\underline c_{jj}, \overline c_{jj}]\times[\underline c_{ij}, \overline c_{ij}]\times[\underline s_{ij}, \overline s_{ij}]$, where $(i,j)$ is an edge in the power network. Specifically, we define the following nonconvex set.
\begin{equation}
\mathcal{K}_{ij}^=:=\left\{(c_{ii},c_{jj}, c_{ij}, s_{ij})\in\mathcal{D}_{ij}: c_{ij}^2+s_{ij}^2  = c_{ii}c_{jj}   \right\}	\label{cone box cons}.
\end{equation}

\paragraph{Convex hull description of $\mathcal{K}_{ij}^=$.}
It follows from \cite[Theorem 1]{Tawarmalani2016} that
\begin{equation}
\text{conv}(\mathcal{K}_{ij}^=) = \mathcal{K}_{ij}^\le \cap \text{conv}(\mathcal{K}_{ij}^\ge),
\end{equation}
where $\mathcal{K}_{ij}^\le$ and $\mathcal{K}_{ij}^\ge$ are defined accordingly. First of all, $\mathcal{K}_{ij}^\le$ is a convex set (in fact, second-order cone representable as it is the rotated cone in $\mathbb{R}^4$). Therefore,  in order to construct $\text{conv}(\mathcal{K}_{ij}^=) $, it suffices to find $\text{conv}(\mathcal{K}_{ij}^\ge)$, which is the intersection of a polytope and the complement of a convex set. Here, we use the following result:

\begin{thm}(Theorem 1 in  \cite{hillestad1980}) \label{thm:poly reverse}
Let $P \subset \mathbb{R}^n$ be a nonempty polytope and $G$ be a proper subset of $\mathbb{R}^n$ such that $\mathbb{R}^n \setminus G$ is convex. Then, $\text{conv}(P \cap G)$ is a polytope.
\end{thm}
The proof of Theorem \ref{thm:poly reverse} is constructive. Let $E_l$, $l=1,\dots,L$ be all the one-dimensional faces of the polytope $P$, i.e. edges of $P$. Then, we have 
\begin{equation}
\text{conv}(P \cap G) = \text{conv}\left(\bigcup_{l=1}^L \text{conv}(E_l  \cap G)\right).
\end{equation}
Hence, $\text{conv}(P \cap G)$ is precisely the convex hull of the extreme points of $ \text{conv}(E_l \cap G)$, $l=1,\dots,L$.

We can apply Theorem  \ref{thm:poly reverse} to our case, where $P=\mathcal{D}_{ij}$ and $G=\{(c_{ii},c_{jj}, c_{ij}, s_{ij})\in \mathbb{R}^4: c_{ij}^2+s_{ij}^2  \ge c_{ii}c_{jj}   \}$. The hypercube $\mathcal{D}_{ij}$ has 32 one-dimensional faces. The extreme points of $\text{conv}\left(\mathcal{K}_{ij}^\ge\right)$ can be easily obtain by fixing three of the variables to one of their bounds. The exact procedures to obtain extreme points are given in Algorithms \ref{alg:extr cii}, \ref{alg:extr cij}, and \ref{alg:extr sij}
(see Appendix~\ref{app:extrpoint}).

After we have computed the extreme points of one-dimensional faces, say $z_{ij}^k$, $k=1,\dots,K$, where $z_{ij}:=(c_{ii},c_{jj}, c_{ij}, s_{ij})$ and $K\le 32$, we can give the convex hull description of $\mathcal{K}_{ij}^\ge$ as follows:
\begin{equation}
\text{conv}\left(\mathcal{K}_{ij}^\ge\right) = \text{conv}\biggl(\{z_{ij}^k\}_{k=1}^K\biggr) = \biggl\{z_{ij} : \exists \lambda,\; z_{ij} = \sum_{k=1}^K \lambda_k z_{ij}^k, \sum_{k=1}^K \lambda_k = 1, \lambda \ge 0 \biggr\}.
\end{equation}

The above discussion can be summarized as the following result:
\begin{prop}\label{prop convex hull}
Let $z_{ij}^k$ be computed using Algorithms \ref{alg:extr cii}-\ref{alg:extr sij}. Then,
\begin{align}
\text{conv}(\mathcal{K}_{ij}^=) = \left\{z_{ij} \in \mathcal{D}_{ij} : \exists \lambda : c_{ij}^2+s_{ij}^2 \le c_{ii}c_{jj},  \ z_{ij} = \sum_{k=1}^K \lambda_k z_{ij}^k, \; \sum_{k=1}^K \lambda_k = 1, \lambda \ge 0 \right\}.\label{type 1 minor convexhull}
\end{align} 
\end{prop}

\paragraph{An Outer Approximation to $\text{conv}(\mathcal{K}_{ij}^=)$.} Note that Proposition \ref{prop convex hull} describes the convex hull of $\text{conv}(\mathcal{K}_{ij}^=)$  in an extended space of $(z,\lambda)$, and the dimension of the $\lambda$ variables, $K$, can be as large as $32$. Directly incorporating the full convex hull description \eqref{type 1 minor convexhull} for each edge in the power network could lead to a very large formulation. Instead, we propose an outer approximation to $\text{conv}(\mathcal{K}_{ij}^=)$ in the space of the $z$ variable using only four linear inequalities. Our extensive experiments show that this approximation is quite accurate and computationally efficient.

We again focus on $\text{conv}(\mathcal{K}_{ij}^\ge)$ and rewrite the ``reverse-cone" constraint as follows:
\begin{equation}
f(c_{ij},s_{ij}) := \sqrt{c_{ij}^2+s_{ij}^2} \ge \sqrt{c_{ii}c_{jj}} =: g(c_{ii}, c_{jj}).
\end{equation}
Note that $f$ is a convex function and $g$ is a concave function. If we overestimate the former and underestimate the latter by hyperplanes, the inequality still holds. The following propositions formalize this idea:


\begin{prop}\label{prop cc}
The affine functions $g_m(c_{ii},c_{jj}):=\nu_{ij}^m+\eta_{ij}^m c_{ii} + \delta_{ij}^m c_{jj}$, $m=1,2$, underestimate $\sqrt{c_{ii}c_{jj}}$ over the box $[\underline c_{ii}, \overline c_{ii}]\times[\underline c_{jj}, \overline c_{jj}]$, where

$\eta_{ij}^1 = \frac{\sqrt{\underline c_{ii}}}{\sqrt{\underline c_{jj}}+\sqrt{\overline c_{jj}}}$,
$\delta_{ij}^1 = \frac{\sqrt{\underline c_{jj}}}{\sqrt{\underline c_{ii}}+\sqrt{\overline c_{ii}}}$,
$\nu_{ij}^1 = \sqrt{ \underline c_{ii} \underline c_{jj}} - \eta_{ij}^1 \underline c_{ii} - \delta_{ij}^1 \underline c_{jj}$
and

$\eta_{ij}^2 = \frac{\sqrt{\overline c_{ii}}}{\sqrt{\underline c_{jj}}+\sqrt{\overline c_{jj}}}$,
$\delta_{ij}^2 = \frac{\sqrt{\overline c_{jj}}}{\sqrt{\underline c_{ii}}+\sqrt{\overline c_{ii}}}$,
$\nu_{ij}^2 = \sqrt{ \overline c_{ii} \overline c_{jj}} - \eta_{ij}^2 \overline c_{ii} - \delta_{ij}^2 \overline c_{jj}$.
\end{prop}
\begin{prop}\label{prop cs1}
If $\sqrt{\overline c_{ij}^2 + \overline s_{ij}^2}+\sqrt{\underline c_{ij}^2 + \underline s_{ij}^2}-\sqrt{\overline c_{ij}^2 + \underline s_{ij}^2}-\sqrt{\underline c_{ij}^2 + \overline s_{ij}^2} < 0$, then the affine functions  $f_n(c_{ij}, s_{ij}):=\nu_{ij}^n+\eta_{ij}^n c_{ij} + \delta_{ij}^l s_{ij}$, $n=3,4$, overestimate $ \sqrt{c_{ij}^2+s_{ij}^2}$ over the box $[\underline c_{ij}, \overline c_{ij}]\times[\underline s_{ij}, \overline s_{ij}]$, where

$\eta_{ij}^3 = \frac{ \sqrt{\overline c_{ij}^2 + \underline s_{ij}^2} - \sqrt{\underline c_{ij}^2 + \underline s_{ij}^2} }{ \overline c_{ij} - \underline c_{ij}}$,
$\delta_{ij}^3 = \frac{ \sqrt{ \underline c_{ij}^2 + \overline s_{ij}^2} - \sqrt{\underline c_{ij}^2 + \underline s_{ij}^2} }{ \overline s_{ij} - \underline s_{ij}}$,
$\nu_{ij}^3 = \sqrt{\underline c_{ij}^2 + \underline s_{ij}^2} - \eta_{ij}^3 \underline c_{ij} - \delta_{ij}^3 \underline s_{ij}$
and

$\eta_{ij}^4 = \frac{ \sqrt{\overline c_{ij}^2 + \overline s_{ij}^2} - \sqrt{\underline c_{ij}^2 + \overline s_{ij}^2} }{ \overline c_{ij} - \underline c_{ij}}$,
$\delta_{ij}^4 = \frac{ \sqrt{ \overline c_{ij}^2 + \overline s_{ij}^2} - \sqrt{\overline c_{ij}^2 + \underline s_{ij}^2} }{ \overline s_{ij} - \underline s_{ij}}$,
$\nu_{ij}^4 = \sqrt{\overline c_{ij}^2 + \overline s_{ij}^2} - \eta_{ij}^4 \overline c_{ij} - \delta_{ij}^4 \overline s_{ij}$.
\end{prop}
\begin{prop}\label{prop cs2}
If $\sqrt{\overline c_{ij}^2 + \overline s_{ij}^2}+\sqrt{\underline c_{ij}^2 + \underline s_{ij}^2}-\sqrt{\overline c_{ij}^2 + \underline s_{ij}^2}-\sqrt{\underline c_{ij}^2 + \overline s_{ij}^2} > 0$, then the affine functions  $f_n(c_{ij},s_{ij}):=\nu_{ij}^n+\eta_{ij}^n c_{ij} + \delta_{ij}^n s_{ij}$, $n=3,4$, overestimate $ \sqrt{c_{ij}^2+s_{ij}^2}$ over the box $[\underline c_{ij}, \overline c_{ij}]\times[\underline s_{ij}, \overline s_{ij}]$, where

$\eta_{ij}^3 = \frac{ \sqrt{\overline c_{ij}^2 + \overline s_{ij}^2} - \sqrt{\underline c_{ij}^2 + \overline s_{ij}^2} }{ \overline c_{ij} - \underline c_{ij}}$,
$\delta_{ij}^3 = \frac{ \sqrt{ \underline c_{ij}^2 + \overline s_{ij}^2} - \sqrt{\underline c_{ij}^2 + \underline s_{ij}^2} }{ \overline s_{ij} - \underline s_{ij}}$,
$\nu_{ij}^3 = \sqrt{\underline c_{ij}^2 + \underline s_{ij}^2} - \eta_{ij}^3 \underline c_{ij} - \delta_{ij}^3 \underline s_{ij}$
and

$\eta_{ij}^4 = \frac{ \sqrt{\overline c_{ij}^2 + \underline s_{ij}^2} - \sqrt{\underline c_{ij}^2 + \underline s_{ij}^2} }{ \overline c_{ij} - \underline c_{ij}}$,
$\delta_{ij}^4 = \frac{ \sqrt{ \overline c_{ij}^2 + \overline s_{ij}^2} - \sqrt{\overline c_{ij}^2 + \underline s_{ij}^2} }{ \overline s_{ij} - \underline s_{ij}}$,
$\nu_{ij}^4 = \sqrt{\overline c_{ij}^2 + \overline s_{ij}^2} - \eta_{ij}^4 \overline c_{ij} - \delta_{ij}^4 \overline s_{ij}$.
\end{prop}
\begin{prop}
Let $\nu_{ij}^m$, $\eta_{ij}^m$ and $\delta_{ij}^m$, $m=1,2$ and $\nu_{ij}^n$, $\eta_{ij}^n$ and $\delta_{ij}^n$, $n=3,4$ be calculated using Propositions \ref{prop cc}-\ref{prop cs2}. Then, the edge cuts (EC) defined as
\begin{equation}\label{edge cut 4 together}
\text{EC}_{ij}(\underline c, \overline c,\underline s,\overline s): \ 
\nu_{ij}^n+\eta_{ij}^n c_{ij} + \delta_{ij}^n s_{ij} \ge \nu_{ij}^m+\eta_{ij}^m c_{ii} + \delta_{ij}^m c_{jj}, \ m=1,2 \text{ and } n=3,4
\end{equation}
are valid for $\text{conv}(\mathcal{K}_{ij}^=)$.
\end{prop}

A different analysis related to Type 1 minor condition is carried out in \cite{Chen2015Report} by considering the following set:
\begin{equation}
\begin{split}
\mathcal{K}_{ij}' =  \{(c_{ii},c_{jj}, c_{ij}, s_{ij}): c_{ij}^2+s_{ij}^2  = c_{ii}c_{jj}, 
(c_{ii},c_{jj}) \in [\underline c_{ii}, \overline c_{ii}]\times[\underline c_{jj}, \overline c_{jj}], \\ c_{ij} \tan\underline \theta_{ij} \le s_{ij} \le c_{ij}\tan\overline \theta_{ij} \}.
\end{split}
\end{equation}
It turns out that conv$(\mathcal{K}_{ij}')$ is second-order cone representable with  two non-trivial linear inequalities. In our experiments, we observe that the addition of these valid linear inequalities does not produce any extra gap closure in our approach, and hence, they are not used in our relaxation scheme.

\subsubsection{Types 2 and 3: 3- and 4-Cycle Minors}
\label{type 2 and 3, cycle relax}
The real and imaginary parts of Type 2 and Type  3 minors in \eqref{type 2 minor} and \eqref{type 3 minor} can be written generically as $ \sum_{i=1}^N a_i x_i y_i = 0$ for some  $a \in \mathbb{R}^N $ with $a_i \neq 0$. Let $\underline x, \overline x, \underline y, \overline y$ be $N$-vectors with the property that $\underline x < \overline x $ and $\underline y < \overline y$.
We are interested in finding the convex hull of the following set:
\begin{equation} 
\mathcal{S}_a = \bigg \{(x,y) \in \mathbb{R}^N\times\mathbb{R}^N :  \sum_{i=1}^N a_i x_i y_i = 0, \  \underline x \le x \le \overline x, \ \underline y \le y \le \overline y \bigg \}.
\end{equation}

We have the following theorem, whose proof is given in Appendix \ref{app:main theorem 1 proof}:
\begin{thm}\label{main theorem 1}
{conv}$(\mathcal{S}_a)$ is second-order cone representable.
\end{thm}
In the proof of Theorem \ref{main theorem 1}, we construct the convex hull of $\text{conv}(S_a)$  in an extended space with the number of disjunctions exponential in $N$, where some of the disjunctions may contain second-order conic constraints. Including the exact description of $\text{conv}(S_a)$ as part of the relaxation might be quite costly. In the following, we propose tight outer approximation of $\text{conv}(S_a)$.

\paragraph{An Outer Approximation to $\text{conv}(S_a)$.} 
We propose a linear outer approximation to $\text{conv}(S_a)$ using McCormick envelopes  \cite{mccormick} as follows:
\begin{equation} 
\begin{split} 
\mathcal{S}_a^M &= \{(x,y) \in \mathbb{R}^N\times\mathbb{R}^N : \exists w \in \mathbb{R}^N : \sum_{i=1}^N a_i w_i = 0, \\ & \quad \  \max\{\underline y x + \underline x y - \underline x \underline y, \overline y x + \overline x y - \overline x \overline y \} \le w \le \min\{\underline y x + \overline x y - \overline x \underline y, \overline y x + \underline x y - \underline x \overline y \}
 \}.
\end{split}
\end{equation}
Our extensive tests show that $\mathcal{S}_a^M$ tightly approximates conv$(\mathcal{S}_a)$.


\paragraph{A Cycle Based Relaxation} Although we have linear outer approximations of Type 2 and 3 minors, their total number is cubic and quartic in the number of buses. Therefore, including relaxations for every minor can be quite expensive. Instead, we focus on a set of cycles in the graph and include a subset of minors for each cycle. This corresponds to triangulating a given cycle into 3- and 4-cycles. This idea is similar to the one used in our previous paper \cite{kocuk2015} although the construction in that paper is entirely different.

We do not propose to include all such minors in our relaxation scheme either. Rather, we construct a cycle relaxation and use it as a basis to generate cutting planes. More precisely, let $C$ be a given cycle in the power network and $X^C$ denote the principal submatrix of $X$, which corresponds to the rows and columns of the nodes in the cycle $C$. Next, we choose a subset of Type 2 and 3 minors from $X^C$ and define the following set
\begin{equation}
\mathcal{Q}_C = \{(c,s): \exists (\tilde c, \tilde s) : q_k(c,s,\tilde c, \tilde s)=0, \ k \in \mathcal{K}_C\},
\end{equation}
where $q_k$'s  are bilinear equations corresponding to minor constraints \eqref{type 2 minor cs} and \eqref{type 3 minor cs} index by $\mathcal{K}_C$. Here, we denote the variables associated with original edge $(i,j)$ in the power network as $c_{ij},s_{ij}$, and denote  $\tilde c_{ij}, \tilde s_{ij}$ for an artificial edge added to triangulate the cycle $C$. See Figure \ref{triangulation example} for an illustration on how a 7-cycle can be triangulated.

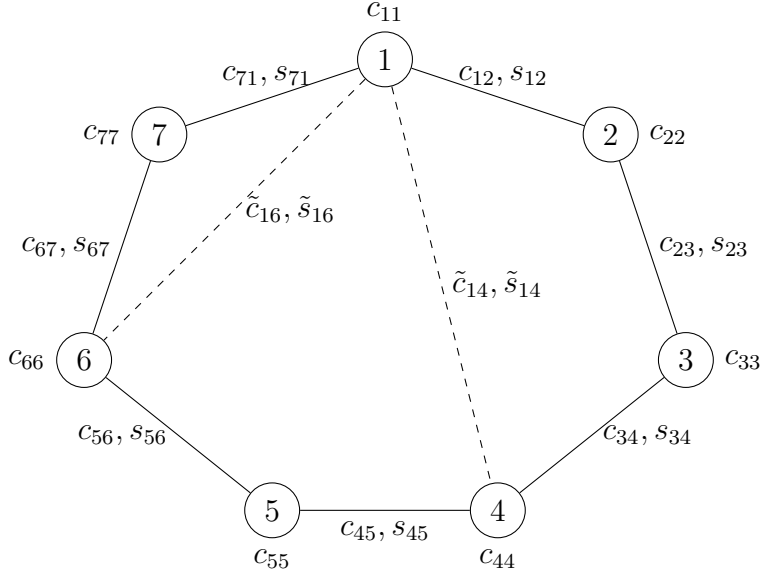
\begin{figure*}[!htp]
\centering
\begin{tikzpicture}
    \node[shape=circle,draw=black, label = {$c_{11}$}] (A) at (0,0) {1};
    \node[shape=circle,draw=black, label = right:{$c_{22}$}] (B) at (3,-1) {2};
    \node[shape=circle,draw=black, label = right:{$c_{33}$}] (C) at (4,-4) {3};
    \node[shape=circle,draw=black, label = below:{$c_{44}$}] (D) at (1.5,-6) {4};
    \node[shape=circle,draw=black, label = below:{$c_{55}$}] (E) at (-1.5,-6) {5};
    \node[shape=circle,draw=black, label = left:{$c_{66}$}] (F) at (-4,-4) {6} ;
    \node[shape=circle,draw=black, label = left:{$c_{77}$}] (G) at (-3,-1) {7} ;

    \path [-] (A) edge node [align=center,above]{{ }$c_{12}, s_{12}$} (B);
    \path [-] (B) edge node [align=center,right]{$c_{23}, s_{23}$} (C);
    \path [-] (C) edge node [align=center,right]{$c_{34}, s_{34}$} (D);
    \path [-] (D) edge node [align=center,below]{$c_{45}, s_{45}$} (E);
    \path [-] (E) edge node [align=center,left]{$c_{56}, s_{56}$} (F);
    \path [-] (F) edge node [align=center,left]{$c_{67}, s_{67}$} (G);
    \path [-] (G) edge node [align=center,above]{$c_{71}, s_{71}${ }} (A);
    \path [dashed] (A) edge node [align=center,right]{$\tilde c_{14}, \tilde s_{14}$} (D);
    \path [dashed] (A) edge node [align=center,right]{$\tilde c_{16}, \tilde s_{16}$} (F);

\end{tikzpicture}
\caption[Triangulation of a 7-cycle.]{
Triangulation of a 7-cycle. Here, we  consider three minors corresponding to the ``subcycles" $\{1,2,3,4\}$, $\{1,4,5,6\}$ and $\{1,6,7\}$. Variables corresponding to each node and edge are also shown.
}\label{triangulation example}
\end{figure*}

Finally, we write the McCormick relaxation for the nonconvex set $\mathcal{Q}_C$, parametrized by the variable bounds $\underline c, \overline c,\underline s, \overline s$,  for a given cycle $C$, compactly, as follows:
\begin{equation}\label{lp system cs}
\mathcal{M}_C (\underline c, \overline c,\underline s, \overline s)  = \bigg \{(c,s): \exists (\tilde c, \tilde s, w) :  A \begin{bmatrix}  c \\  s \end{bmatrix} + \tilde A  \begin{bmatrix} \tilde c \\ \tilde s \end{bmatrix}  + B w \le c,  \ E w = 0 \bigg \}.
\end{equation}
Here, $w$ is a vector of new variables defined to linearize the bilinear terms in the minor constraints. Inequality constraints  contain the McCormick envelopes of the bilinear terms and bounds  on the variables, while equality constraints include the linearized minor equality constraints. 

\paragraph{Discretization}
\label{paragraph:discretize}

Our preliminary analysis has shown that $\mathcal{M}_C$ does not approximate conv$(\mathcal{Q}_C)$ accurately in most cases. However, it is well-known that McCormick relaxation of a nonconvex set  converges to the convex hull of the set as the variable ranges shrink. Motivated by this fact, we propose a discretization technique to improve the accuracy of the set $\mathcal{M}_C$. We note that this idea has been applied to other nonconvex problems in the literature (e.g. pooling problem \cite{Misener, DeyGu,GupteAhDeCh}).

To begin with, let $[\underline c_d, \overline c_d] \times [\underline s_d, \overline s_d]$, for $d \in D$, be a partition of the initial box $[\underline c, \overline c] \times [\underline s, \overline s]$. Then, the following relations trivially hold:
\begin{equation} \label{def disc McC}
\mathcal{M}_C^D := \text{conv} \left ( \bigcup_{d \in D} \mathcal{M}_C (\underline c_d, \overline c_d,\underline s_d, \overline s_d) \right )   \subseteq \mathcal{M}_C (\underline c, \overline c,\underline s, \overline s)  
\quad \text{ and } \quad 
\text{conv}(\mathcal{Q}_D) \subseteq \mathcal{M}_C^D.
\end{equation}
Note that since $\bigcup_{d \in D} \mathcal{M}_C (\underline c_d, \overline c_d,\underline s_d, \overline s_d) $ is a finite union of polyhedral representable sets, $\mathcal{M}_C^D$ is also polyhedral representable.

In our implementation, we use the following construction. First, we decide a set of variables and bisect their variable ranges to obtain the collection $D$. Then, we construct the set $\mathcal{M}_C^D$. Finally, we solve the separation problem presented in Section \ref{sec:sep prob} to obtain cutting planes
valid for $\mathcal{M}_C^D$.

After initial calibration, we  decided to choose the collection $D$ as follows. We first  choose a reference node to start triangulation. Then, for each subcycle of a cycle, we pick the edges which are neither the first nor the last line in a subcycle and apply bisection to the corresponding $c_{ij}$ and $s_{ij}$ variables.  For instance, in Figure \ref{triangulation example}, variables $c_{23}, s_{23},c_{34}, s_{34},c_{45}, s_{45},c_{56}, s_{56},c_{67}, s_{67}$ are bisected.



%
%
%
%
%
%

\subsection{An Alternative to Type 2 and 3 Minor Conditions: Arctangent Constraints}
\label{sec:2by2 minors arctan}

In this section, we propose another equivalent characterization of the rank-one, or equivalently $2\times2$ minors requirement explained above. Our alternative condition is based on the following relationship between the phase angles $\theta$ and $c,s$ variables, which correspond to the real and imaginary parts of the complex matrix variable $X$ using the $\atan2(y,x)$\footnotemark\footnotetext{$\operatorname{atan2}(y, x) = \begin{cases}
\arctan\frac y x & \quad x > 0 \\
\arctan\frac y x + \pi& \quad y \ge 0 , x < 0 \\
\arctan\frac y x - \pi& \quad y < 0 , x < 0 \\
+\frac{\pi}{2} & \quad y > 0 , x = 0 \\
-\frac{\pi}{2} & \quad y < 0 , x = 0 \\
\text{undefined} & \quad y = 0, x = 0
\end{cases}$} function:
\begin{equation}
\theta_j - \theta_i = \atan2({s_{ij}},{c_{ij}})   \quad i,j \in \mathcal{B} \label{arctangent cons}
\end{equation}

\subsubsection{Equivalence}
First, we claim that Type 1 minor constraint \eqref{type 1 minor} together with arctangent constraint \eqref{arctangent cons} implies Type 2 minor \eqref{type 2 minor} equations.

\begin{prop}\label{prop type 2 = type 1 + arctan}
\begin{equation*}
\begin{split}
\bigg\{(c,s):  \begin{vmatrix} X_{ii} & X_{ij} \\ X_{ki} & X_{kj} \end{vmatrix} = 0  \bigg\} \supseteq 
\bigg\{(c,s): \exists \theta:  \begin{vmatrix} X_{ii} & X_{ij} \\ X_{ji} & X_{jj} \end{vmatrix} = 
			         \begin{vmatrix}  X_{kk} & X_{ki}   \\X_{ik}& X_{ii}\end{vmatrix} = 
			         \begin{vmatrix} X_{kk} & X_{kj} \\ X_{jk} & X_{jj} \end{vmatrix} = 0, \\
				\theta_j - \theta_i = \atan2({s_{ij}},{c_{ij}}),
				\theta_i - \theta_k = \atan2({s_{ki}},{c_{ki}}),
				\theta_j - \theta_k = \atan2({s_{kj}},{c_{kj}}) \bigg \}.
\end{split}
\end{equation*}
\end{prop}
\begin{proof}
From \eqref{arctangent cons}, we have that $c_{ij}=\sqrt{c_{ii}c_{jj}}\cos(\theta_j-\theta_i)$ and $s_{ij}=\sqrt{c_{ii}c_{jj}}\sin(\theta_j-\theta_i)$.
Then, 
\begin{equation*}
\begin{split}
c_{ii}c_{kj} - c_{ij}c_{ki}  + s_{ij}s_{ki} =& c_{ii}\sqrt{c_{kk}c_{jj}}\cos(\theta_j-\theta_k) 
- \sqrt{c_{ii}c_{jj}}\cos(\theta_j-\theta_i)\sqrt{c_{kk}c_{ii}}\cos(\theta_i-\theta_k) \\ 
&  \hspace{3.9cm} +\sqrt{c_{ii}c_{jj}}\sin(\theta_j-\theta_i)\sqrt{c_{kk}c_{ii}}\sin(\theta_i-\theta_k) \\
= & c_{ii}\sqrt{c_{jj}c_{kk}}[\cos(\theta_j-\theta_k) - \cos(\theta_j-\theta_i)\cos(\theta_i-\theta_k) + \sin(\theta_j-\theta_i)\sin(\theta_i-\theta_k) ]\\
=&0,
\end{split}
\end{equation*}
and
\begin{equation*}
\begin{split}
c_{ii}s_{kj} - s_{ij}c_{ki}  - c_{ij}s_{ki} =& c_{ii}\sqrt{c_{kk}c_{jj}}\sin(\theta_j-\theta_k) 
- \sqrt{c_{ii}c_{jj}}\sin(\theta_j-\theta_i)\sqrt{c_{kk}c_{ii}}\cos(\theta_i-\theta_k) \\ 
&  \hspace{3.9cm} -\sqrt{c_{ii}c_{jj}}\cos(\theta_j-\theta_i)\sqrt{c_{kk}c_{ii}}\sin(\theta_i-\theta_k) \\
= & c_{ii}\sqrt{c_{jj}c_{kk}}[\sin(\theta_j-\theta_k) - \sin(\theta_j-\theta_i)\cos(\theta_i-\theta_k) - \cos(\theta_j-\theta_i)\sin(\theta_i-\theta_k) ]\\
=&0,
\end{split}
\end{equation*}
which imply Type 2 minor \eqref{type 2 minor}.
\end{proof}

A similar proposition about Type 3 minors is also true, that is, Type 1 minor constraint \eqref{type 1 minor} together with arctangent constraint \eqref{arctangent cons} implies Type 3 minor \eqref{type 3 minor} equations.  

\begin{prop}\label{prop type 3 = type 1 + arctan}
\begin{equation*}
\begin{split}
\bigg\{(c,s):  \begin{vmatrix} X_{ij} & X_{ik} \\ X_{lj} & X_{lk} \end{vmatrix}  = 0  \bigg\} \supseteq
\bigg\{(c,s): \exists \theta:  \begin{vmatrix} X_{ii} & X_{ij} \\ X_{ji} & X_{jj} \end{vmatrix} = 
			         \begin{vmatrix} X_{ii} & X_{ik} \\ X_{ki} & X_{kk} \end{vmatrix} = 
			         \begin{vmatrix} X_{jj} & X_{jl} \\ X_{lj} & X_{ll} \end{vmatrix} = 0,  
			         \begin{vmatrix} X_{kk} & X_{kl} \\ X_{lk} & X_{ll} \end{vmatrix} = 0, \\
				\theta_j - \theta_i = \atan2({s_{ij}},{c_{ij}}),
				\theta_k - \theta_i = \atan2({s_{ik}},{c_{ik}}),
				\theta_j - \theta_l = \atan2({s_{lj}},{c_{lj}}) ,
				\theta_l - \theta_l = \atan2({s_{lk}},{c_{lk}})\bigg \}.
\end{split}
\end{equation*}
\end{prop}
We omitted the proof of Proposition \ref{prop type 3 = type 1 + arctan} due to its similarity to the proof of Proposition~\ref{prop type 2 = type 1 + arctan}.

\subsubsection{Convexification}
In Section \ref{type 2 and 3, cycle relax}, we analyzed Type 2 and 3 minors and proposed a method to obtain a linear outer-approximation. Now, we propose  another linearization method using the arctangent restriction \eqref{arctangent cons}. To start with, let us define the following nonconvex set
\begin{equation}
\mathcal{AT}:=\left\{(c,s,\theta)\in\mathbb{R}^3 : \theta = \arctan \left ( \frac sc \right), (c,s,\theta)\in[\underline c, \overline c]\times[\underline s, \overline s] \times [\underline \theta, \overline \theta] \right\},	\label{arctan constraint}
\end{equation}
where we denote $\theta= \theta_j - \theta_i$ and drop $(i,j)$ indices for brevity. We also assume $\underline c>0$. Let us denote the four corners of the box in $(c,s,\theta)$ space as follows: 
\begin{equation}
\begin{split}
   \zeta^1 &= (\underline c, \overline s, \arctan \left( {\overline s}/{\underline c} \right) ), \quad
    \zeta^2 = (\overline c, \overline s, \arctan \left( {\overline s}/{\overline c} \right)), \\ 
   \zeta^3 &= (\overline c,  \underline s, \arctan \left( {\underline s}/{\overline c} \right)), \quad
   \zeta^4 = (\underline c, \underline s, \arctan \left( {\underline s}/{\underline c} \right)). 
\end{split}
\end{equation}
Two inequalities that approximate the upper envelope of $\mathcal{AT}$ are described below.
\begin{prop}
Let $\theta = \gamma_1 + \alpha_1 c + \beta_1 s$  and $\theta = \gamma_2 + \alpha_2 c + \beta_2 s$ be the  planes passing through points $\{\zeta^1,\zeta^2, \zeta^3\}$,  and $\{\zeta^1,\zeta^3, \zeta^4\}$, respectively.
Then, two valid inequalities for $\mathcal{AT}$ can be obtained as
\begin{align}\label{eq:arctan-upper}
\bar \gamma_m + \alpha_m c + \beta_m s \ge \arctan \left( \frac{s}{c} \right)
\end{align}
for all $(c,s) \in[\underline c, \overline c]\times[\underline s, \overline s]$ with $\bar \gamma_m = \gamma_m + \Delta \gamma_m$, where
\begin{equation} \label{error problem with angle}
\Delta \gamma_m =   \max \left \{  \arctan \left( \frac{s}{c} \right) - (\gamma_m + \alpha_m c + \beta_m s ) :   (c, s) \in [\underline c, \overline c] \times [\underline s, \overline s],  \ c \tan\underline \theta \le s \le c\tan\overline \theta \right \}, 
\end{equation} 
for $m=1,2$.
\end{prop}
Note that by the construction of \eqref{error problem with angle}, it is evident that $\bar \gamma_m+\alpha_m c + \beta_m s$ dominates the $\arctan(s/c)$ over the box. The nonconvex optimization problem  \eqref{error problem with angle} can be solved by enumerating all possible Karush-Kuhn-Tucker (KKT) points. See Appendix \ref{app:arctan} for details. These inequalities are  improvements over the similar ones in \cite{kocuk2015} since the bounds on $\theta$ variables are also taken into consideration in the calculation of the offset value $\Delta \gamma_m$.

Two inequalities that approximate the lower envelope of $\mathcal{AT}$ are described below.

\begin{prop}
Let $\theta = \gamma_3 + \alpha_3 c + \beta_3 s$  and $\theta = \gamma_4 + \alpha_4 c + \beta_4 s$ be the  planes passing through points $\{\zeta^1,\zeta^2, \zeta^4\}$,  and $\{\zeta^2,\zeta^3, \zeta^4\}$, respectively.
Then, two valid inequalities for  $\mathcal{AT}$ are defined as
\begin{align}\label{eq:arctan-lower}
 \bar \gamma_n + \alpha_n c + \beta_n s \le \arctan \left( \frac{s}{c} \right)
\end{align}
for all $(c,s) \in[\underline c, \overline c]\times[\underline s, \overline s]$ with $\bar \gamma_n = \gamma_n - \Delta \gamma_n$, where 
\begin{equation} \label{eq:arctan-lower-deltagamma}
\Delta \gamma_n =   \max \left \{ (\gamma_n+ \alpha_n c + \beta_n s ) - \arctan \left( \frac{s}{c} \right)   :    (c, s) \in [\underline c, \overline c] \times [\underline s, \overline s],  \ c \tan \underline \theta \le s \le c\tan\overline \theta \right \}, 
\end{equation} 
for $n=3,4$.
\end{prop}

In summary, the four linear inequalities that approximate conv$(\mathcal{AT}_{ij})$ are given as
\begin{equation}\label{arctan 4 together}
\text{AT}_{ij}(\underline c, \overline c,\underline s,\overline s): \ 
\bar \gamma_{ij}^m + \alpha_{ij}^m c_{ij} + \beta_{ij}^m s_{ij} \le  \theta_j - \theta_i  \le \bar \gamma_{ij}^n + \alpha_{ij}^n c_{ij} + \beta_{ij}^n s_{ij} , \ m= 1,2 \text{ and } n = 3,4
\end{equation}
for some line $(i,j)$.

\subsection{General Principal Submatrices}
\label{sec: sdp separate}


Up until this point, we have mainly discussed how to incorporate the minor constraints into our relaxation scheme. In this section, we propose an approach to include a relaxed version of the positive semidefiniteness constraint by considering general principal minors of the matrix variable $X$. Our approach is a relaxation of the positive semidefiniteness requirement since  $X\succeq 0$ implies that \textit{all} principal submatrices should be positive semidefinite while we only include a few hyperplanes which outer-approximate $X^s\succeq 0$ for some principal submatrix $X^s$.
From a different view point, the approach in this section can  be seen as a simultaneous convexification of several, appropriately chosen Type 2 and 3 minor conditions \eqref{type 2 minor cs} and \eqref{type 3 minor cs}.

Let $\mathcal{B}'$ be a subset of buses. Let $x \in \mathbb{R}^{2|\mathcal{B}'|}$ be a vector of bus voltages defined as $x=[e^T; f^T]^T$ such that $x_i=e_i$ for $i\in\mathcal{B}'$ and $x_{i'}=f_i$ for $i'=i+|\mathcal{B}'|$. Observe that the following linear relationship between $c$, $s$ and $W$ holds, 
\begin{subequations}\label{eq:sdplin W}
\begin{align}
c_{ij} &= e_ie_j+f_if_j= W_{ij} + W_{i'j'} & i&,j\in {\mathcal{B}'}\\ 
s_{ij} &= e_if_j-e_jf_i= W_{ij'} - W_{ji'} & i&,j\in {\mathcal{B}'} \\ 
c_{ii} &= e_i^2+f_i^2= W_{ii} + W_{i'i'},   & i&\in\mathcal{B}'\\
W &= xx^T.
\end{align}
\end{subequations}
Here, we used real matrices instead of complex matrices for convenience. It is proven in \cite{taylorbook} that using real matrices is equivalent to complex matrices.

Clearly, the set defined by \eqref{eq:sdplin W} is nonconvex. A straightforward SDP relaxation can be presented as follows:
\begin{equation} 
\begin{split}
\mathcal{S}_{\mathcal{B}'}:=\{(c,s) \in \mathbb{R}^{2|\mathcal{B}'|} : \exists  W\in\mathbb{S}^{2|\mathcal{B}'|}_+,\;
c_{ij} =  W_{ij} + W_{i'j'} \quad i,j\in {\mathcal{B}'}\\ 
s_{ij} = W_{ij'} - W_{ji'}   \quad i,j\in {\mathcal{B}'} \\ 
c_{ii} = W_{ii} + W_{i'i'}  \quad i \in\mathcal{B}' \
 \}.
\end{split}
\end{equation}

Note that this relaxation is a further relaxation of the SDP relaxation since only one principal submatrix is considered here. Although the subset $\mathcal{B}'$ can be general, previous experience \cite{kocuk2015} shows that it makes sense to use a subset of buses that correspond to a cycle. In particular, we define the following set $\mathcal{S}_C$ for a cycle $C$
\begin{equation}\label{eq:sdplin W relax}
\begin{split}
\mathcal{S}_C:=\{(c,s) \in \mathbb{R}^{2|C|} : \exists  W\in\mathbb{S}^{2|C|}_+,\;
&c_{ij} =  W_{ij} + W_{i'j'} \quad (i,j)\in C\\ 
&s_{ij} = W_{ij'} - W_{ji'}   \quad (i,j)\in C\\ 
&c_{ii} = W_{ii} + W_{i'i'}  \quad \ i \in \{ k: (k,l) \in C \} \
 \},
\end{split}
\end{equation}
and then, use the following procedure  in Section \ref{sec:sep prob}  to obtain cutting planes for this set.

%

\subsection{Separation Problems for LP and SDP Based Cycle Relaxations} 
\label{sec:sep prob}

In Sections \ref{type 2 and 3, cycle relax} and  \ref{sec: sdp separate}, we presented  LP and SDP based approximations for the cycles in our problem. In this section, we propose to utilize these relaxations in a cutting plane framework.

Suppose that we are given a solution $(c^*, s^*)$ and we want to determine whether this point belongs to either LP or SDP relaxation given in 
\eqref{lp system cs} or \eqref{eq:sdplin W relax}. In order to cover both cases, let us focus on a generic setting where we have  a conic representable set ${S} = \{x : \exists u, Ax  + B u\succeq_{K} b \}$ and a point $x^*$ that we want to separate. Here, we assume that the cone $K$ is either the nonnegative orthant or the cone of positive semidefinite matrices with appropriate dimension. 
{Suppose we want to determine if the given point $x^*$ is in ${S}$, or find a separating hyperplane $\alpha^T x\ge\beta$ such that $\alpha^T x^* < \beta$ otherwise. This problem can be formulated as 
$$
\max_{\alpha,\beta}\left\{\beta-\alpha^T x^* : \alpha^T x\geq \beta \; \forall x\in S, \; \|\alpha\|_\infty\le 1 \right\}.
$$
Clearly, the optimal value of this problem without the norm constraint is unbounded if $x^* \notin S$. In order to find a cutting plane in the proposed form, we dualize the constraint and normalize the $\alpha$ vector to be in the unit $\ell_\infty$ ball. The resulting problem of separating a given point $x^*$ from a conic representable set $S$  is given as follows:
\begin{align}\label{general sep prob}
\hspace{-2mm}\text{SEP}(S, x^*): Z^*:=\max_{\alpha,\beta,\mu}\{ \beta-\alpha^T x^* : b^T\mu\geq\beta, A^T\mu=\alpha, B^T\mu=0,
                        \mu\in K, -e \le \alpha \le e \}.
\end{align}
Here, we have two cases: If $Z^*\leq 0$, then $x^*\in S$, otherwise, the optimal solution $(\alpha^*,\beta^*)$ from \eqref{general sep prob} gives
a separating hyperplane of the form $\alpha^T x\ge\beta$, which then can be used as a cutting plane.

\subsection{Bound Tightening}

So far, one of the standing assumptions for the construction of McCormick relaxations and arctangent outer-approximation was the availability of lower and upper bounds on $c$ and $s$ variables. Clearly, tighter variable bounds will lead to better relaxations. In this section, we explain how good bounds can be obtained by first solving small size bounding SOCPs and then, improving these bounds further by incorporating some dual information.

\subsubsection{Optimization-Based Bound Tightening}
Assuming some angle bounds for a line $(i,j)\in \mathcal{L}$ as
\begin{equation}
\underline{\theta}_{ij} \le \theta_i-\theta_j \le \overline{\theta}_{ij},
\end{equation}
we can obtain a rough first estimate for the bounds on $c$ and $s$ as follows:
\begin{subequations}
\begin{align}
\underline c_{ij} &= \underline V_i \underline V_j \cos(\underline \theta_{ij}),   &\overline c_{ij}& = \overline V_i \overline V_j ,
 \\
\underline s_{ij} &= \overline V_i \overline V_j \sin(\underline \theta_{ij}),  &\overline s_{ij}& = \overline V_i \overline V_j  \sin(\overline \theta_{ij}).
\end{align}
\end{subequations}
We claim that these bounds can be further tightened by solving  SOCP bounding problems.  Let  us consider a line $(k,l) \in \mathcal{L}$ and fix a ``closeness" parameter $r$. We first define  the following sets: $\mathcal{B}_{kl} (r)$, the set of buses which can be reached from either $k$ or $l$ in at most $r$ steps, and $\mathcal{L}_{kl} (r) $, the set of lines incident to at least one bus in $\mathcal{B}_{kl} (r)$.
%
{Consider} the  following second-order cone representable set, 
\begin{subequations}\label{bounding SOCP}
\begin{align}
 &\hspace{0.5em} p_i^g-p_i^d = G_{ii}c_{ii} + \sum_{j \in \delta(i)}[ G_{ij}c_{ij} -B_{ij}s_{ij}]   & i& \in\mathcal{B}_{kl} (r),\\
  & \hspace{0.5em} q_i^g-q_i^d = -B_{ii}c_{ii} + \sum_{j \in \delta(i)}[ -B_{ij}c_{ij} -G_{ij}s_{ij}]  & i& \in \mathcal{B}_{kl} (r), \\
  & \hspace{0.5em} \underline V_i^2 \le c_{ii} \le \overline V_i^2    & i& \in \mathcal{B}_{kl} (r+1), \\
  & \hspace{0.5em}  p_i^{\text{min}}  \le p_i^g \le p_i^{\text{max}}     & i& \in \mathcal{B}_{kl} (r), \\
  & \hspace{0.5em}  q_i^{\text{min}}  \le q_i^g \le q_i^{\text{max}}     & i& \in \mathcal{B}_{kl} (r), \\
  & \hspace{0.5em} c_{ij}=c_{ji}, \ \ s_{ij}=-s_{ji}    &(&i,j) \in \mathcal{L}_{kl} (r),\\
  & \hspace{0.5em}  [-G_{ij}c_{ii} + G_{ij}c_{ij} -B_{ij}s_{ij}]^2+[B_{ij}c_{ii}-B_{ij}c_{ij} -G_{ij}s_{ij}]^2  \le (S_{ij}^{\text{max}})^2  &(&i,j) \in \mathcal{L}_{kl} (r), \label{powerOnArcReform  bounding prob} \\
  & \hspace{0.5em}  c_{ij}^2+s_{ij}^2  \le c_{ii}c_{jj}     &(&i,j) \in \mathcal{L}_{kl} (r), \\
  & \hspace{0.5em} \text{EC}_{ij}(\underline c, \overline c,\underline s,\overline s) &(&i,j) \in \mathcal{L}_{kl} (r), \label{edge minor ineq bound}  \\
  & \hspace{0.5em} \text{AT}_{ij}(\underline c, \overline c,\underline s,\overline s) &(&i,j) \in \mathcal{L}_{kl} (r), \label{ineq arctan bound}  \\
  & \hspace{0.5em}\underline{\theta}_{ij} \le \theta_i-\theta_j \le \overline{\theta}_{ij} &(&i,j) \in \mathcal{L}_{kl} (r), \label{angle theta bounds bounding prob}  \\
  & \hspace{0.5em} \underline c_{ij} \le c_{ij}  \le \overline c_{ij}, \ \ \underline s_{ij} \le s_{ij}  \le \overline s_{ij}   &(&i,j) \in \mathcal{L}_{kl} (r),
\end{align}
\end{subequations}
where EC and AT are defined as in \eqref{edge cut 4 together} and \eqref{arctan 4 together}, respectively.
This set is an improved version of the one proposed in \cite{kocuk2015}, which does not contain edge cut inequalities \eqref{edge minor ineq bound} for Type 1 minors or arctangent envelopes \eqref{ineq arctan bound} .

Let us now define the following problems:
\begin{equation}\label{bounding SOCP all}
\begin{split}
\underline P_{kl}^{c} (\underline c, \overline c,\underline s, \overline s, r):  &\hspace{0.5em} \underline c_{kl}^* =\min\{  c_{kl} :  \eqref{bounding SOCP} \}, \\
\overline P_{kl}^{c} (\underline c, \overline c,\underline s, \overline s, r):    &\hspace{0.5em} \overline c_{kl}^* =\max\{  c_{kl} : \eqref{bounding SOCP} \}, \\
\underline P_{kl}^{s}(\underline c, \overline c,\underline s, \overline s, r):  &\hspace{0.5em} \underline s_{kl}^* =\min\{  s_{kl} :  \eqref{bounding SOCP} \}, \\
\overline P_{kl}^{s} (\underline c, \overline c,\underline s, \overline s, r):    &\hspace{0.5em} \overline s_{kl}^* = \max\{ s_{kl} :  \eqref{bounding SOCP} \}.
\end{split}
\end{equation}  
For improved numerical stability, we update a variable bound only if it is improved by at least $10^{-3}$.
As an implementation note, since these problems are independent of each other for different edges, they can be solved in parallel. According to our experiments, this synchronous parallelization saves a significant amount of computational time.

For artificial edges, it is not possible to use the above procedure as they do not appear in the flow balance constraints. However, we can utilize the bounds already computed for the original edges to obtain some bounds for the variables defined for  the artificial edges by adopting the procedure proposed in \cite{kocuk2015}.

\subsubsection{Dual-Based Bound Tightening}

Let us suppose that the problems \eqref{bounding SOCP all} have been solved and we have updated the variable bounds on $c$ and $s$ variables to $\underline{c}^*$, $\overline{c}^*$, $\underline{s}^*$ and $\overline{s}^*$. Note that while solving these problems, the existing variable bounds are used, that is, the bounds are not updated. A simple way to incorporate the change in one problem to another is to use the dual variables. Related methods have been used by LP based global solvers~\cite{ts:05}. Let us now formally explain how this can be accomplished.


Consider the problem $\underline P_{kl}^c$.  Let $\underline \pi_{kl}^c(\underline c_{ij})$, $\underline \pi_{kl}^c(\overline c_{ij})$, $\underline \pi_{kl}^c(\underline s_{ij})$ and
 $\underline \pi_{kl}^c(\overline s_{ij})$ be the optimal dual variables  corresponding to the constraints 
$c_{ij} \ge \underline c_{ij}$, $c_{ij} \le \overline c_{ij}$,  $s_{ij} \ge \underline s_{ij}$ and $s_{ij} \le \overline s_{ij}$, respectively.
First, we calculate the contribution of the constraints other than the bounds on the dual objective as
\begin{equation} 
\Pi_{kl}^c =  \underline c_{kl}^* - \sum_{(i,j) \in \in \mathcal{L}_{kl} (r)} \left( 
\underline c_{ij}\underline \pi_{kl}^c(\underline c_{ij})  + \overline c_{ij}\underline \pi_{kl}^c(\overline c_{ij}) + 
 \underline s_{ij}\underline \pi_{kl}^c(\underline s_{ij})+ \overline s_{ij} \underline \pi_{kl}^c(\overline s_{ij})
\right).
\end{equation}

Now, since the bounds on $c_{ij}$ and $s_{ij}$ variables are improved via their own bounding problems $\underline P_{ij}^{c}$, $\overline P_{ij}^{c}$,  $\underline P_{ij}^{s}$ and $\overline P_{ij}^{s}$, the lower bound on $c_{kl}$ can be updated as follows:
\begin{equation}
\underline c_{kl}^* = \max \bigg\{ \underline c_{kl}^*, \ \Pi_{kl}^c + \sum_{(i,j) \in \in \mathcal{L}_{kl} (r)} \left( 
\underline c_{ij}^*\underline \pi_{kl}^c(\underline c_{ij})  + \overline c_{ij}^*\underline \pi_{kl}^c(\overline c_{ij}) + 
 \underline s_{ij}^*\underline \pi_{kl}^c(\underline s_{ij})+ \overline s_{ij}^* \underline \pi_{kl}^c(\overline s_{ij})
\right) \bigg\}.
\end{equation}
Similarly, using the dual variables from  $\overline P_{kl}^{c}$,  $\underline P_{kl}^{s}$ and $\overline P_{kl}^{s}$, we may try to tighten $\overline c_{kl}^*$, $\underline s_{kl}^*$ and $\overline s_{kl}^*$ further.

\section{SOCP Based Spatial Branch-and-Cut Method}
\label{sec:alg socp}

In Section \ref{sec:rank constrained}, we proposed several convexification techniques for the minor (or equivalently, rank) constrained OPF problem including convex and concave envelopes and cutting planes. In this section, we will show how they can be used in a relaxation scheme. In Section \ref{sec:alg socp root}, we develop an algorithm which can be used  stand-alone as a cutting plane approach to find dual bounds for the OPF problem. It can also be treated as the root node relaxation of the SOCP based spatial branch-and-cut algorithm proposed in Section \ref{sec:alg socp bb}. Finally,  Section \ref{sec:alg implement} presents the implementation details of the spatial branch-and-cut algorithm.

\subsection{Root Node Relaxation}
\label{sec:alg socp root}

The computational experiments will be based on an SOCP relaxation of the OPF problem. Let $\text{SOCP}(\underline c, \overline c,\underline s, \overline s, \mathcal{H})$ denote this relaxation constructed by using the variable bounds $\underline c$, $\overline c$, $\underline s$ and $\overline s$, and a set of cutting planes of the form $\alpha_h^T\begin{bmatrix} c \\ s \end{bmatrix}  \ge \beta_h $ from an index set $h\in\mathcal{H}$ obtained by solving  separation problems. The full model is defined as follows:
\begin{subequations} \label{SOCP relax}
\begin{align}
& \hspace{-0.7cm} \text{SOCP}(\underline c, \overline c,\underline s, \overline s, \mathcal{H}): \notag \\
\min  &\hspace{0.5em}  \sum_{i \in \mathcal{G}} C_i(p_i^g) \\
  \mathrm{s.t.}   
 &\hspace{0.5em} p_i^g-p_i^d = G_{ii}c_{ii} + \sum_{j \in \delta(i)}[ G_{ij}c_{ij} -B_{ij}s_{ij}]   & i& \in\mathcal{B},\\
  & \hspace{0.5em} q_i^g-q_i^d = -B_{ii}c_{ii} + \sum_{j \in \delta(i)}[ -B_{ij}c_{ij} -G_{ij}s_{ij}]  & i& \in \mathcal{B}, \\
  & \hspace{0.5em} \underline V_i^2 \le c_{ii} \le \overline V_i^2    & i& \in \mathcal{B}, \\
  & \hspace{0.5em} c_{ij}=c_{ji}, \ \ s_{ij}=-s_{ji}    &(&i,j) \in \mathcal{L}\\
  & \hspace{0.5em}  [-G_{ij}c_{ii} + G_{ij}c_{ij} -B_{ij}s_{ij}]^2+[B_{ij}c_{ii}-B_{ij}c_{ij} -G_{ij}s_{ij}]^2  \le (S_{ij}^{\text{max}})^2  &(&i,j) \in \mathcal{L}, \label{powerOnArcReform} \\
  & \hspace{0.5em}  c_{ij}^2+s_{ij}^2  \le c_{ii}c_{jj}     &(&i,j) \in \mathcal{L}, \\
  & \hspace{0.5em} \text{EC}_{ij}(\underline c, \overline c,\underline s,\overline s) &(&i,j) \in \mathcal{L}, \label{edge minor ineq}  \\
  & \hspace{0.5em} \text{AT}_{ij}(\underline c, \overline c,\underline s,\overline s) &(&i,j) \in \mathcal{L}, \label{ineq arctan}  \\
  & \hspace{0.5em}\underline{\theta}_{ij} \le \theta_i-\theta_j \le \overline{\theta}_{ij} &(&i,j) \in \mathcal{L}, \label{angle theta bounds}  \\
  & \hspace{0.5em} \underline c_{ij} \le c_{ij}  \le \overline c_{ij}, \ \ \underline s_{ij} \le s_{ij}  \le \overline s_{ij}   &(&i,j) \in \mathcal{L}, \\
  & \hspace{0.5em} \alpha_h^T\begin{bmatrix} c \\ s \end{bmatrix}  \ge \beta_h  &h& \in \mathcal{H}, \\
   & \hspace{0.5em}  \eqref{activeAtGenerator} - \eqref{reactiveAtGenerator}.  \notag
\end{align}
\end{subequations}

Our approach heavily depends on tightening the variable bounds and enriching the set of cutting planes so that SOCP relaxation gets tightened. The main steps of this root node relaxation algorithm  is summarized in Algorithm \ref{alg:root}.
\begin{algorithm}
\caption{Root node relaxation.}
\label{alg:root}
\begin{algorithmic}
\STATE $LB = -\infty$,  $UB = \infty$, $\mathcal{C}=\emptyset$, $\mathcal{H}=\emptyset$, $t = 0$
\STATE Use a local solver to find a feasible solution and update $UB$.
\STATE Obtain a cycle basis $C_b$ and set $\mathcal{C} = C_b$.
\STATE Solve bound tightening problems $\underline P_{kl}^{c}$, $\overline P_{kl}^{c}$, $\underline P_{kl}^{s}$ and $\overline P_{kl}^{s}$ for all $(k,l)\in \mathcal{L}$ with $r_1$ and apply dual improvement.
\WHILE{$t < T$ and $LB < (1-\epsilon)UB$}
\IF{$t < T_C$ and $|\mathcal{B}| \le B_{\text{max}}$}
\STATE Enlarge $\mathcal{C}$ by adding more cycles.
\ENDIF
\STATE Solve bound tightening problems $\underline P_{kl}^{c}$, $\overline P_{kl}^{c}$, $\underline P_{kl}^{s}$ and $\overline P_{kl}^{s}$ for all $(k,l)\in \mathcal{L}$ with $r_2$ and apply dual improvement.
\STATE Solve $\text{SOCP}(\underline c, \overline c,\underline s, \overline s, \mathcal{H})$ to obtain a solution $(c^*,s^*)$ and update $LB$.
\STATE Solve $\text{SEP}(\mathcal{S}_C, c^*,s^*)$ and/or $\text{SEP}(\mathcal{M}_C^D, c^*,s^*)$ to obtain a set of cutting planes $H_t$ for all $C \in \mathcal{C}$.
\STATE Update $\mathcal{H}= \mathcal{H} \cup H_t$.
\STATE Set $t = t+1$.
\ENDWHILE
\end{algorithmic}
\end{algorithm}

For small ($|\mathcal{B}| \le B_{\text{max}}$) and challenging instances, we generate new cycles for $T_{C}$ many rounds using every pair of distinct cycles in the current set $\mathcal{C}$  which share at least one common line.

\subsection{Spatial Branch-and-Cut Algorithm}
\label{sec:alg socp bb}

Algorithm \ref{alg:root} is quite successful in proving strong dual bounds for many  instances from the NESTA archive as the numerical experiments in Section \ref{results:root node relax} show. Nevertheless, the optimality gap may be more than an acceptable threshold for some of the more challenging instances, for which we propose an  SOCP based spatial branch-and-cut algorithm. The main steps can be seen in Algorithm \ref{alg:sbb}.

Our approach is built on the following principles:
\begin{enumerate}
\item  Branching: In our approach, we decide a transmission line $(i,j)$ and branch on either $c_{ij}$ and $s_{ij}$. This branching rule allows us to update convex approximations to both EC$_{ij}$ and AT$_{ij}$. We pick the line to be branched on node $L$ of the branch-and-bound tree as follows:
\begin{equation}
line_L = \max_{(i,j) \in \mathcal{L}} \left |\theta_j-\theta_i-\arctan \left( \frac{s_{ij}}{c_{ij}} \right) \right|.
\end{equation}
Then, among $c_{ij}$ and $s_{ij}$, we choose the variable whose smallest distance to the boundary is the largest. In particular, if $\min\{c_{ij} - \underline c_{ij}, \overline c_{ij} -  c_{ij} \} \ge  \min\{s_{ij} - \underline s_{ij}, \overline s_{ij} -  s_{ij} \}$, then $c_{ij}$ is chosen; otherwise, $s_{ij}$ is chosen.
 Finally, we use bisection-branching to partition the space \cite{BARON}.

\item Local bound tightening: Since branching on a variable $c_{ij}$ or $s_{ij}$ reduces the variable range, other variables which correspond to the nearby lines to the branched line can be improved as well. Therefore, we solve the bound tightening problems for such lines in our algorithm.

\item Node selection: Since our aim is to reduce the duality gap on the problem, we choose the node with the smallest node relaxation value and carry out the branching.

\item Cutting plane generation: We keep on generating cutting planes to separate relaxation solutions. To be computationally efficient, we only solve the separating problems for  the cycles at hand which contains the branched line.
\end{enumerate}

\begin{algorithm}
\caption{Spatial branch-and-cut.}
\label{alg:sbb}
\begin{algorithmic}
\STATE Let $LB$,  $UB$, $\mathcal{C}$ and $\mathcal{H}$ be computed from Algorithm \ref{alg:root}. 
\STATE Set $list=\{root\}$.

\WHILE{$|list| > 0$}
\STATE $LB = \min_{ l \in list }LB_l$ and $L = \argmin _{ l \in list }LB_l$.
\STATE $list = list \setminus \{L\}$.
\IF{$LB \ge (1-\epsilon)UB$}
\STATE STOP.
\ENDIF

\STATE Solve bound tightening problems $\underline P_{kl}^{c}$, $\overline P_{kl}^{c}$, $\underline P_{kl}^{s}$ and $\overline P_{kl}^{s}$ for all $(k,l)$ near  ${line_{parent(L)}}$ with $r_2$ and apply dual improvement.
\STATE Solve $\text{SEP}(\mathcal{S}_C, c^*,s^*)$ and/or $\text{SEP}(\mathcal{M}_C^D, c^*,s^*)$ to obtain a set of cutting planes $H_t$ for all $C \in \mathcal{C}$ such that $line_{parent(L)} \in C$.
\STATE Update $\mathcal{H}_L = \mathcal{H}_L \cup {H}_t$.
\STATE Solve $\text{SOCP}(\underline c, \overline c,\underline s, \overline s, \mathcal{H})$ to obtain a solution $(c^*,s^*)$ and update $LB$.

\STATE Decide on a transmission line $line_L$ to branch on.
\STATE Obtain two children $L_1$ and  $L_2$ by updating variable bounds, EC and AT.
\STATE $list = list \cup \{L_1, L_2\}$.

\ENDWHILE

\end{algorithmic}
\end{algorithm}

\subsection{Implementation} 
\label{sec:alg implement}

In Algorithm \ref{alg:sbb}, SOCP relaxation of each node can be constructed from scratch given the following four pieces of information:
\begin{enumerate}
\item variable bounds,
\item its parent's relaxation solution,
\item transmission line   branched on, and
\item the valid inequalities of its parent.
\end{enumerate}
Therefore, a direct implementation can be obtained by explicit tree handling as long as the parent inherits this set of information.

This implementation is a reasonable attempt since, unlike LPs, there is no efficient warm-start availability for SOCPs. There are also some disadvantages: For instance, the proposed implementation requires the construction of each problem from scratch and explicit tree handling. Although the data needed to be stored at each node is limited, there may be some issues for large problems.

In this implementation, the overhead is the solution of SOCPs at each node of the branch-and-bound tree. We prefer to use MOSEK in this implementation since it is an efficient conic interior point solver.

Finally,  bound tightening and separation problems are parallelized to reduce the  total computational time.

\section{Computational Experiments}
\label{sec: comp exper}

In this section, we present the results of our extensive computational experiments from NESTA 0.3.0 archive \cite{nesta} with Typical, Congested and Small Angle Operating Conditions. We are particularly interested in this set of instances due to their difficulty level, as explained below. Our main code is written in the C\# language with Visual Studio 2010 as the compiler. For comparison purposes, we use OPF Solver \cite{OPFSolver} to solve the SDP relaxation of the OPF problem. This MATLAB package exploits sparsity of the power networks to efficiently solve large-scale SDP problems  \cite{madani2014, madani2015}. We modified the code slightly to incorporate phase angle difference constraints.
 For all experiments, we used a 64-bit computer with Intel Core i5 CPU 2.50GHz processor and 16 GB RAM.  Time is measured in seconds, unless otherwise stated. Conic interior point solver MOSEK 8 \cite{MOSEK} is used to solve LPs, SOCPs and SDPs in our main algorithms. OPF Solver is run with MOSEK and SDPT3.

\subsection{Methods}

We run our algorithms with different settings as to cutting plane generation procedures:
\begin{itemize}
\item
McCormick Separation ($\mathsf{SEP(M)}$): We only separate the point from $\mathcal{M}_C^D$ defined in \eqref{def disc McC}.
\item
SDP Separation ($\mathsf{SEP(S)}$): We only separate the point from $\mathcal{S}_C$ defined in \eqref{eq:sdplin W relax}.
\item
SDP + McCormick Separation ($\mathsf{SEP(M, S)}$): We  separate the point from $\mathcal{M}_C^D$ and~$\mathcal{S}_C$.
\end{itemize}
We use a fixed cycle basis to generate cutting planes for most instances. For small ($|\mathcal{B}| \le B_{\text{max}} = 118$) and difficult instances, we enlarge the set of cycles to obtain more cuts.
After initial calibration, we decide to set the number of bound tightening rounds $T$ to 5, the number of cycle addition rounds $T_C$ to 1,  the initial radius $r_1$ to 2, the later radius $r_2$ to 4, and the optimality tolerance $\epsilon$  to $10^{-3}$.
We also employ coefficient rounding for SDP cuts to improve numerical stability.

The OPF Solver code is modified to incorporate the phase angle bounds in NESTA instances by adding the following constraints:
\begin{equation}
 \Im(X_{ij}) - \tan \overline \theta_{ij} \Re(X_{ij}) \le 0 \quad \text{ and } \quad 
 \Im(X_{ij}) - \tan \underline \theta_{ij} \Re(X_{ij}) \ge 0.
\end{equation}
We run the OPF Solver with two solvers:
\begin{itemize}
\item MOSEK 
\item SDPT3 
\end{itemize}

We also compare the three approaches proposed in this paper to our previous paper \cite{kocuk2015}, which utilizes strong SOCP relaxations. For Typical and Congested Operating Conditions, we use a simplified version of  $\mathsf{SEP(S)}$ and for Small Angle Operating Condition, we again use a simpler version of $\mathsf{SEP(M)}$, which proved to be the best setting in that paper. 



\subsection{Comparison to the SOCP and SDP Relaxation}

We compared the relaxation values obtained from our approach to the plain SOCP, strong SOCP, and SDP relaxations in terms of the optimality gap, which is calculated as follows: $\text{\%gap} = 100 \times \frac{z^{\text{UB}} - z^{\text{LB}}}{z^{\text{UB}}}$. Here, $z^{\text{LB}}$ is the optimal objective cost of a relaxation and $z^{\text{UB}}$ is the objective cost of a  feasible solution obtained by MATPOWER \cite{Matpower}.

\subsubsection{Root Node Relaxation}
\label{results:root node relax}

In this section, we present the computational results for Typical, Congested, and Small Angle Operating Condition instances in Tables \ref{table:rootTYP}-\ref{table:rootSAD}, respectively. Also, we provide a scatter plot Figure \ref{figure:root results}, which visualizes the average percentage optimality gap and computational times.

\begin{landscape}
\begin{table}
\centering
\begin{tabular}{|c|r|r|r|r|r|r|r|r|r|r|r|r|}
\hline
           & \multicolumn{ 2}{|c|}{SOCP} & \multicolumn{ 2}{|c|}{SDP} & \multicolumn{ 2}{|c|}{Best of \cite{kocuk2015}} & \multicolumn{ 2}{|c|}{$\mathsf{SEP(M)}$} & \multicolumn{ 2}{|c|}{$\mathsf{SEP(S)}$} & \multicolumn{ 2}{|c|}{$\mathsf{SEP(M, S)}$} \\
\hline
case           &      \%gap &   time (s) &      \%gap &   time (s) &      \%gap &   time (s) &      \%gap &   time (s) &      \%gap &   time (s) &      \%gap &   time (s) \\
\hline
     3lmbd &       1.32 &       0.06 &       0.39 &       1.00 &       0.43 &       0.14 &       0.10 &       1.09 &       0.09 &       1.12 &       0.10 &       0.95 \\
\hline
       4gs &       0.00 &       0.05 &       0.00 &       0.98 &       0.00 &       0.08 &       0.00 &       0.11 &       0.00 &       0.05 &       0.00 &       0.03 \\
\hline
      5pjm &      14.54 &       0.09 &       5.22 &       1.03 &       6.22 &       0.17 &       5.63 &       1.59 &       3.68 &       2.04 &       2.11 &       3.26 \\
\hline
       6ww &       0.63 &       0.02 &       0.00 &       1.30 &       0.00 &       0.44 &       0.02 &       0.51 &       0.01 &       0.75 &       0.01 &       1.08 \\
\hline
     9wscc &       0.00 &       0.06 &       0.00 &       1.03 &       0.00 &       0.08 &       0.00 &       0.09 &       0.00 &       0.06 &       0.00 &       0.09 \\
\hline
    14ieee &       0.11 &       0.05 &       0.00 &       1.36 &       0.00 &       0.53 &       0.03 &       0.80 &       0.00 &       1.31 &       0.00 &       2.70 \\
\hline
    29edin &       0.14 &       0.11 &       0.00 &       2.98 &       0.00 &       1.81 &       0.06 &       8.67 &       0.01 &      21.06 &       0.01 &      33.99 \\
\hline
      30as &       0.06 &       0.03 &       0.00 &       2.49 &       0.00 &       0.90 &       0.06 &       0.06 &       0.06 &       0.14 &       0.06 &       0.11 \\
\hline
     30fsr &       0.39 &       0.06 &       0.00 &       1.93 &       0.03 &       0.92 &       0.10 &       9.61 &       0.07 &       8.58 &       0.07 &      14.49 \\
\hline
    30ieee &      15.65 &       0.03 &       0.00 &       1.50 &       0.00 &       0.92 &       0.09 &      10.55 &       0.03 &      10.05 &       0.03 &      14.55 \\
\hline
    39epri &       0.05 &       0.09 &       0.01 &       2.21 &       0.01 &       0.67 &       0.05 &       0.14 &       0.05 &       0.09 &       0.05 &       0.25 \\
\hline
    57ieee &       0.06 &       0.06 &       0.00 &       3.05 &       0.00 &       1.73 &       0.06 &       0.11 &       0.06 &       0.25 &       0.06 &       0.22 \\
\hline
   118ieee &       2.10 &       0.17 &       0.07 &       6.31 &       0.25 &       4.67 &       0.42 &     115.50 &       0.14 &     228.60 &       0.14 &     355.50 \\
\hline
   162ieee &       4.19 &       0.17 &       1.12 &      17.93 &       3.50 &       9.19 &       2.14 &     373.38 &       1.56 &     666.08 &       1.57 &     948.30 \\
\hline
   189edin &       0.22 &       0.28 &       0.07 &       6.59 &       0.08 &       2.04 &       0.28 &      48.30 &       0.10 &      87.17 &       0.04 &      63.15 \\
\hline
   300ieee &       1.19 &       0.39 &       0.08 &      16.36 &       0.30 &       9.41 &       0.23 &     321.62 &       0.09 &     345.69 &       0.09 &     520.50 \\
\hline
    2383wp &       1.68 &       6.37 &       0.37 &     840.31 &       1.56 &      74.25 &       1.48 &      33.59 &       1.36 &      46.10 &       1.16 &     163.27 \\
\hline
    2736sp &       1.57 &       8.21 &       0.00 &    1265.13 &       1.42 &      91.07 &       1.82 &      49.62 &       1.07 &      90.00 &       0.87 &     165.61 \\
\hline
   2737sop &       6.54 &       4.68 &       0.00 &    1228.51 &       1.57 &      87.30 &       5.35 &      51.83 &       1.15 &     185.08 &       1.15 &     348.35 \\
\hline
   2746wop &      13.61 &       3.98 &       0.00 &    1329.97 &       1.50 &      91.39 &       2.40 &      66.19 &       2.57 &     140.07 &       2.57 &     273.07 \\
\hline
    2746wp &       2.48 &       6.44 &       0.00 &    1383.23 &       1.54 &      91.84 &       2.61 &      41.28 &       1.08 &      61.44 &       1.04 &     107.09 \\
\hline
\hline
   Average &       3.17 &       1.50 &       0.35 &     291.20 &       0.88 &      22.36 &       1.09 &      54.03 &       0.63 &      90.27 &       0.53 &     143.65 \\
\hline
\end{tabular}  
\caption{Root node relaxation results for Typical Operating Condition instances.}\label{table:rootTYP}
\end{table}
\end{landscape}

\begin{landscape}
\begin{table}
\centering
\begin{tabular}{|c|r|r|r|r|r|r|r|r|r|r|r|r|}
\hline
           & \multicolumn{ 2}{|c|}{SOCP} & \multicolumn{ 2}{|c|}{SDP} & \multicolumn{ 2}{|c|}{Best of \cite{kocuk2015}} & \multicolumn{ 2}{|c|}{$\mathsf{SEP(M)}$} & \multicolumn{ 2}{|c|}{$\mathsf{SEP(S)}$} & \multicolumn{ 2}{|c|}{$\mathsf{SEP(M, S)}$} \\
\hline
case           &      \%gap &   time (s) &      \%gap &   time (s) &      \%gap &   time (s) &      \%gap &   time (s) &      \%gap &   time (s) &      \%gap &   time (s) \\
\hline
     3lmbd &       3.30 &       0.05 &       1.26 &       0.93 &       1.31 &       0.14 &       0.81 &       0.62 &       0.78 &       1.68 &       0.81 &       1.05 \\
\hline
       4gs &       0.65 &       0.02 &       0.00 &       0.88 &       0.00 &       0.09 &       0.05 &       0.78 &       0.03 &       0.47 &       0.03 &       0.55 \\
\hline
      5pjm &       0.45 &       0.05 &       0.45* &       1.03 &       0.00 &       0.20 &       0.09 &       0.41 &       0.05 &       0.76 &       0.05 &       0.81 \\
\hline
       6ww &      13.33 &       0.03 &       0.00 &       1.04 &       0.00 &       0.48 &       0.00 &       1.18 &       0.00 &       1.87 &       0.00 &       3.39 \\
\hline
     9wscc &       0.00 &       0.05 &       0.00 &       0.93 &       0.00 &       0.11 &       0.00 &       0.08 &       0.00 &       0.08 &       0.00 &       0.06 \\
\hline
    14ieee &       1.35 &       0.06 &       0.00 &       1.06 &       0.00 &       0.58 &       0.14 &       8.67 &       0.03 &       7.16 &       0.04 &      13.18 \\
\hline
    29edin &       0.44 &       0.11 &       0.44* &       2.90 &       0.03 &       1.96 &       0.08 &      66.27 &       0.04 &      93.01 &       0.04 &     136.83 \\
\hline
      30as &       4.76 &       0.06 &       0.00 &       1.96 &       1.72 &       1.00 &       0.11 &      25.06 &       0.08 &      37.44 &       0.09 &      62.09 \\
\hline
     30fsr &      45.97 &       0.06 &      11.06 &       1.90 &      40.22 &       1.00 &       9.91 &      27.92 &       5.13 &      50.09 &       5.15 &      90.56 \\
\hline
    30ieee &       0.99 &       0.05 &       0.00 &       2.32 &       0.08 &       0.98 &       0.18 &      24.04 &       0.06 &      37.84 &       0.06 &      60.03 \\
\hline
    39epri &       2.99 &       0.05 &       0.00 &       2.55 &       0.00 &       0.73 &       0.09 &       9.44 &       0.01 &      17.14 &       0.01 &      26.33 \\
\hline
    57ieee &       0.21 &       0.06 &       0.08 &       2.69 &       0.13 &       1.82 &       0.20 &      13.95 &       0.06 &      84.38 &       0.06 &     125.44 \\
\hline
   118ieee &      44.19 &       0.14 &      31.53 &       7.47 &      39.09 &       5.07 &      14.38 &     285.36 &       7.91 &     517.63 &       7.83 &     911.90 \\
\hline
   162ieee &       1.52 &       0.23 &       1.00 &      21.43 &       1.20 &       9.88 &       1.17 &     617.04 &       1.03 &    1393.66 &       1.03 &    2007.66 \\
\hline
   189edin &       5.88 &       0.22 &       0.05 &       6.53 &       3.82 &       2.28 &       1.12 &     212.12 &       0.89 &     444.09 &       0.91 &     592.86 \\
\hline
   300ieee &       0.85 &       0.39 &       0.00 &      14.65 &       0.15 &       9.94 &       0.22 &     571.77 &       0.10 &     735.95 &       0.10 &    1048.07 \\
\hline
    2383wp &       0.89 &       2.07 &       0.10 &     857.87 &       0.00 &      58.46 &       0.99 &      69.62 &       0.40 &     208.76 &       0.40 &     592.17 \\
\hline
    2736sp &       2.13 &       2.73 &       0.07 &    1439.24 &       0.72 &      66.89 &       1.57 &      65.91 &       1.32 &     109.79 &       1.32 &     308.40 \\
\hline
   2737sop &       1.08 &       2.93 &       0.01 &    1203.03 &       0.31 &      63.74 &       1.10 &      40.59 &       0.65 &     268.75 &       0.65 &     636.37 \\
\hline
   2746wop &       0.52 &       2.70 &       0.00 &    1413.02 &       0.00 &      67.87 &       0.51 &      42.65 &       0.42 &     118.05 &       0.42 &     296.78 \\
\hline
    2746wp &       0.59 &       2.96 &       0.00 &    1457.02 &       0.00 &      74.26 &       0.64 &      43.13 &       0.21 &     174.71 &       0.71 &     153.60 \\
\hline
\hline
   Average &       6.29 &       0.72 &       2.19 &     306.69 &       4.23 &      17.50 &       1.59 &     101.27 &       0.91 &     204.92 &       0.94 &     336.58 \\
\hline
\end{tabular}  

\caption{Root node relaxation results for Congested Operating Condition instances. *: Numerical difficulties are encountered for the SDP relaxation, the resuts  from SOCP relaxation is used instead.}\label{table:rootAPI}
\end{table}
\end{landscape}

\begin{landscape}
\begin{table}
\centering
\begin{tabular}{|c|r|r|r|r|r|r|r|r|r|r|r|r|}
\hline
           & \multicolumn{ 2}{|c|}{SOCP} & \multicolumn{ 2}{|c|}{SDP} & \multicolumn{ 2}{|c|}{Best of \cite{kocuk2015}} & \multicolumn{ 2}{|c|}{$\mathsf{SEP(M)}$} & \multicolumn{ 2}{|c|}{$\mathsf{SEP(S)}$} & \multicolumn{ 2}{|c|}{$\mathsf{SEP(M, S)}$} \\
\hline
case           &      \%gap &   time (s) &      \%gap &   time (s) &      \%gap &   time (s) &      \%gap &   time (s) &      \%gap &   time (s) &      \%gap &   time (s) \\
\hline
     3lmbd &       4.28 &       0.05 &       2.06 &       1.15 &       1.52 &       0.10 &       0.11 &       1.02 &       0.28 &       0.64 &       0.09 &       1.29 \\
\hline
       4gs &       4.90 &       0.02 &       0.05 &       1.01 &       0.03 &       0.11 &       0.01 &       0.20 &       0.01 &       0.22 &       0.01 &       0.66 \\
\hline
      5pjm &       3.61 &       0.03 &       0.00 &       1.11 &       0.39 &       0.14 &       0.07 &       0.34 &       0.08 &       0.36 &       0.07 &       0.94 \\
\hline
       6ww &       0.80 &       0.02 &       0.00 &       1.34 &       0.02 &       0.27 &       0.00 &       0.51 &       0.00 &       0.66 &       0.00 &       1.53 \\
\hline
     9wscc &       1.50 &       0.05 &       0.00 &       1.02 &       0.43 &       0.20 &       0.01 &       0.44 &       0.01 &       0.45 &       0.01 &       1.14 \\
\hline
    14ieee &       0.07 &       0.03 &       0.00 &       1.39 &       0.06 &       0.44 &       0.06 &       0.03 &       0.06 &       0.06 &       0.06 &       0.16 \\
\hline
    29edin &      34.47 &       0.09 &      28.44 &       2.51 &      21.92 &       4.79 &       0.90 &      63.05 &       0.80 &     168.42 &       0.70 &     325.68 \\
\hline
      30as &       9.16 &       0.08 &       0.47 &       1.82 &       2.47 &       1.26 &       0.14 &      11.64 &       0.09 &      20.67 &       0.09 &      38.85 \\
\hline
     30fsr &       0.62 &       0.09 &       0.07 &       2.19 &       0.29 &       1.12 &       0.13 &      10.44 &       0.09 &      14.01 &       0.09 &      26.57 \\
\hline
    30ieee &       5.87 &       0.08 &       0.00 &       2.41 &       2.04 &       1.01 &       0.08 &       5.30 &       0.02 &      13.78 &       0.02 &      26.78 \\
\hline
    39epri &       0.11 &       0.03 &       0.09 &       2.25 &       0.09 &       1.02 &       0.03 &       2.67 &       0.02 &       6.01 &       0.02 &      11.54 \\
\hline
    57ieee &       0.11 &       0.08 &       0.02 &       2.92 &       0.10 &       1.83 &       0.08 &       6.07 &       0.07 &      19.81 &       0.07 &      36.75 \\
\hline
   118ieee &      12.88 &       0.17 &       7.55 &       6.03 &       7.41 &       6.20 &       4.57 &     122.62 &       3.35 &     375.91 &       3.35 &     748.42 \\
\hline
   162ieee &       7.06 &       0.14 &       3.56 &      20.66 &       5.86 &      14.30 &       4.12 &     418.86 &       3.77 &    1123.51 &       3.76 &    1741.94 \\
\hline
   189edin &       2.36 &       0.25 &       1.20 &       6.70 &       2.33 &       5.44 &       4.04 &      56.13 &       1.04 &     346.65 &       1.41 &     315.67 \\
\hline
   300ieee &       1.27 &       0.39 &       0.13 &      15.61 &       0.71 &      16.75 &       0.23 &     345.04 &       0.11 &     773.31 &       0.10 &    1226.36 \\
\hline
    2383wp &       5.46 &       5.26 &       1.30 &     850.43 &       3.67 &     554.44 &       3.56 &      86.25 &       3.08 &     142.53 &       3.08 &     550.16 \\
\hline
    2736sp &       3.47 &       6.49 &       0.69 &    1415.42 &       2.02 &     676.52 &       2.77 &      57.89 &       3.07 &     113.58 &       3.81 &     448.04 \\
\hline
   2737sop &       3.63 &       7.27 &       1.00 &    1298.22 &       3.55 &     694.33 &       6.82 &      38.16 &       4.56 &      71.45 &       5.37 &     153.54 \\
\hline
   2746wop &       4.32 &       7.10 &       1.20 &    1448.25 &       3.94 &     772.76 &       6.36 &      40.00 &       3.98 &      67.80 &       4.56 &     150.09 \\
\hline
    2746wp &       3.76 &       6.38 &       0.43 &    1327.80 &       2.73 &     811.51 &       3.50 &      42.57 &       4.20 &     129.86 &       2.74 &     155.73 \\
\hline
\hline
   Average &       5.22 &       1.62 &       2.30 &     305.25 &       2.93 &     169.74 &       1.79 &      62.34 &       1.37 &     161.41 &       1.40 &     283.90 \\
\hline
\end{tabular}

\caption{Root node relaxation results for Small Angle Operating Condition instances.}\label{table:rootSAD}
\end{table}
\end{landscape}


Tables \ref{table:rootTYP}-\ref{table:rootSAD} summarize the results of our three methods applied only to the root node relaxation. We see that $\mathsf{SEP(M)}$ approach is both the most efficient and the  weakest in terms of the optimality gap proven among the three methods whereas $\mathsf{SEP(M, S)}$ approach takes the longest computational times but provides the strongest relaxations overall. The accuracy of $\mathsf{SEP(S)}$ approach is very close to $\mathsf{SEP(M, S)}$  with lower computational cost.


\begin{figure}
\centering
\begin{tikzpicture} 
\begin{axis}[
    xlabel={Average \% Gap},
    ylabel={Average Time (s)},]
\addplot[blue,mark=*,mark options={fill=blue},nodes near coords,only marks,
   point meta=explicit symbolic,
   visualization depends on={value \thisrow{anchor}\as\myanchor},
   every node near coord/.append style={anchor=\myanchor}] table[meta=label] {
x y label anchor
4.89	1.28 {SOCP} east
1.61	301.05 {SDP} west
2.68	69.87   {Best of \cite{kocuk2015}} west
1.49	72.55 {$\mathsf{SEP(M)}$} north
0.97	152.20 {{ }  $\mathsf{SEP(S)}$} north
0.96	254.71 {{ } { } { } $\mathsf{SEP(M, S)}$} north
};
\end{axis}
\end{tikzpicture}
\caption{Scatter plot for root node relaxation results.}\label{figure:root results}
\end{figure}
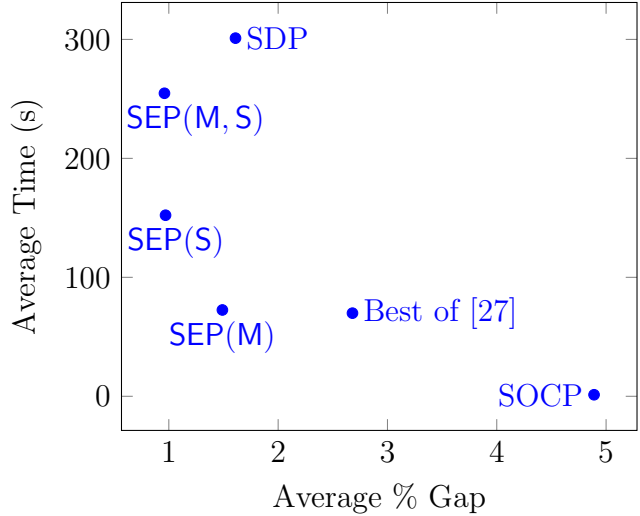

Figure \ref{figure:root results} compares our three methods against SDP Relaxation solved using OPF Solver with SDPT3 chosen as the solver. Although MOSEK is much faster with 172.66 seconds on the average  than SDPT3 with 301.05 seconds, SDPT3 provides more accurate solutions with 1.61\% optimality gap on the average while the optimality gap for MOSEK is 1.98\%. Therefore, we will base our comparisons with SDPT3 results. We can easily see that all our methods, $\mathsf{SEP(M)}$, $\mathsf{SEP(S)}$ and $\mathsf{SEP(M, S)}$, dominate OPF Solver. 
In particular, $\mathsf{SEP(S)}$ and $\mathsf{SEP(M, S)}$ methods are about two times more accurate than OPF Solver in terms of the average optimality gap proven. The computational costs of $\mathsf{SEP(S)}$ and $\mathsf{SEP(M, S)}$ are about 50\% and 15\% less than OPF Solver. 
 We would like to emphasize the success of purely LP and SOCP based method $\mathsf{SEP(M)}$ here. Although it provides the weakest relaxation among our three approaches, it is still  stronger than a purely SDP based approach in about $25\%$ of the computational time. 
In general, we should note that our approaches are much faster on large instances,  more accurate on hard instances, and comparable to the SDP relaxation on small or easy instances.

We also compare the three approaches proposed in this paper to our previous paper \cite{kocuk2015}. The approaches in our previous paper are typically faster than the ones proposed in the current paper, however, the optimality gaps are about 2-3 times worse. This shows the significant improvement from our previous work.


\subsubsection{Effect of Branching}

In this section, we present the detailed  branch-and-cut  results  for  the instances which are not solved within the optimality threshold at the root node relaxation in Table \ref{table:branch results}. We run the branch-and-cut algorithm with a budget of 15 and 30 minutes. The average results are also presented in Table \ref{table:branch results average}. 


Overall, the branching reduces the average optimality gaps from $1.49$, $0.97$ and $0.96$ at the root node to $1.21$, $0.78$  and $0.87$ after 15 minutes for the methods  $\mathsf{SEP(M)}$, $\mathsf{SEP(S)}$ and $\mathsf{SEP(M, S)}$, respectively. There are quite significant gap closure thanks to branching for 5pjm (typical OC instance), 30fsr (congested OC) and 118ieee (congested and small angle OC) instances. 
We should point out that it is possible to process more nodes with $\mathsf{SEP(M)}$ method, which helps to reduce the optimality gaps the most among our three methods, however, it is still the weakest relaxation approach. Also, we can see that $\mathsf{SEP(S)}$ method is able to provide lower  average optimality gap than  $\mathsf{SEP(M, S)}$ at the end since it processes more nodes. 

After 30 minutes, the percentage optimality gaps reduce further to 1.12, 0.71  and 0.82  for the methods  $\mathsf{SEP(M)}$, $\mathsf{SEP(S)}$ and $\mathsf{SEP(M, S)}$, respectively.

\begin{table}\scriptsize
\begin{center}
\begin{tabular}{|c|c|r|r|r|r|r|r|r|r|r|r|r|r|}
\hline
      &        &      \multicolumn{ 4}{|c|}{$\mathsf{SEP(M)}$} &                    \multicolumn{ 4}{|c|}{$\mathsf{SEP(S)}$} &                \multicolumn{ 4}{|c|}{$\mathsf{SEP(M, S)}$} \\
\hline
  \multicolumn{1}{|r|}{OC}        &   Case &      15 m &      30 m & time (s) &       node &      15 m &      30 m &  time (s)  &       node &      15 m &      30 m &  time (s)  &       node \\
\cline{1-14}
\multirow{12}{*}{\begin{sideways}TYPICAL (TYP)\end{sideways}} 
  &   3lmbd &       0.09 &       0.09 &       1.62 &          3 &       0.09 &       0.09 &       1.12 &          1 &       0.09 &       0.09 &       0.95 &          1 \\
\cline{2-14}
  &    5pjm &       0.10 &       0.10 &     124.46 &        270 &       0.10 &       0.10 &      80.23 &        136 &       0.10 &       0.10 &     108.39 &        129 \\
\cline{2-14}
 &    30fsr &       0.10 &       0.10 &      12.18 &          3 &            &            &            &            &            &            &            &            \\
\cline{2-14}
  & 118ieee &       0.26 &       0.24 &    1801.44 &        364 &       0.09 &       0.09 &     324.33 &         13 &       0.10 &       0.10 &     502.47 &         13 \\
\cline{2-14}
&   162ieee &       1.96 &       1.46 &    1811.55 &        139 &       1.56 &       1.33 &    1810.54 &         56 &       1.46 &       1.46 &    1837.44 &         21 \\
\cline{2-14}
 &  189edin &       0.14 &       0.10 &     660.44 &        179 &            &            &            &            &            &            &            &            \\
\cline{2-14}
 &  300ieee &       0.21 &       0.19 &    1807.97 &        287 &            &            &            &            &            &            &            &            \\
\cline{2-14}
&    2383wp &       1.33 &       1.28 &    1811.84 &         43 &       1.16 &       1.08 &    1826.53 &         27 &       0.92 &       0.92 &    1881.99 &         24 \\
\cline{2-14}
&    2736sp &       0.93 &       0.93 &    1810.29 &         50 &       0.87 &       0.87 &    1854.32 &         30 &       0.76 &       0.63 &    1869.73 &         23 \\
\cline{2-14}
 &  2737sop &       5.35 &       5.35 &    1820.06 &         44 &       1.15 &       1.15 &    1800.99 &         27 &       1.15 &       1.15 &    1858.50 &         19 \\
\cline{2-14}
 &  2746wop &       2.40 &       2.40 &    1828.43 &         46 &       2.57 &       2.18 &    1819.52 &         29 &       2.57 &       2.57 &    1819.18 &         21 \\
\cline{2-14}
 &   2746wp &       1.61 &       1.61 &    1828.00 &         38 &       1.04 &       1.04 &    1812.67 &         26 &       2.54 &       2.54 &    1852.72 &         22 \\
\hline
\hline

\multirow{15}{*}{\begin{sideways}CONGESTED (API)\end{sideways}} 
 &    3lmbd &       0.02 &       0.02 &       2.09 &          5 &       0.02 &       0.02 &       3.99 &          5 &       0.02 &       0.02 &       3.34 &          5 \\
\cline{2-14}
 &   14ieee &       0.09 &       0.09 &      11.86 &          5 &            &            &            &            &            &            &            &            \\
\cline{2-14}
 &     30as &       0.06 &       0.06 &      28.84 &          3 &            &            &            &            &            &            &            &            \\
\cline{2-14}
&     30fsr &       4.13 &       1.79 &    1801.24 &        603 &       0.73 &       0.35 &    1803.06 &        393 &       1.31 &       0.83 &    1802.18 &        220 \\
\cline{2-14}
 &   30ieee &       0.09 &       0.09 &      29.33 &          5 &            &            &            &            &            &            &            &            \\
\cline{2-14}
&    57ieee &       0.10 &       0.10 &     226.32 &         57 &            &            &            &            &            &            &            &            \\
\cline{2-14}
&   118ieee &      13.57 &      12.14 &    1803.17 &        115 &       7.83 &       6.17 &    1809.53 &         49 &       7.83 &       7.83 &    1834.74 &         23 \\
\cline{2-14}
&   162ieee &       1.17 &       1.16 &    1831.99 &         48 &       1.03 &       1.03 &    1821.67 &          6 &       1.03 &       1.03 &    2007.68 &          0 \\
\cline{2-14}
&   189edin* &       0.12 &       0.12 &     472.70 &         42 &       0.89 &       0.89 &     543.14 &          9 &       0.12 &       0.12 &     663.19 &          5 \\
\cline{2-14}
&   300ieee &       0.22 &       0.18 &    1810.43 &        100 &            &            &            &            &            &            &            &            \\
\cline{2-14}
&    2383wp &       0.98 &       0.98 &    1806.07 &         24 &       0.20 &       0.20 &    1801.00 &         20 &       0.40 &       0.40 &    1885.56 &         12 \\
\cline{2-14}
 &   2736sp &       1.57 &       1.57 &    1802.19 &         36 &       1.26 &       1.21 &    1825.24 &         18 &       1.32 &       1.26 &    1869.59 &         12 \\
\cline{2-14}
&   2737sop &       1.10 &       1.10 &    1808.90 &         32 &       0.65 &       0.65 &    1809.25 &         18 &       0.65 &       0.65 &    1835.00 &         10 \\
\cline{2-14}
&   2746wop &       0.51 &       0.51 &    1852.43 &         31 &       0.42 &       0.41 &    1852.74 &         25 &       0.42 &       0.36 &    1931.52 &         16 \\
\cline{2-14}
&    2746wp &       0.64 &       0.64 &    1847.24 &         27 &       0.20 &       0.20 &    1811.35 &         19 &       0.71 &       0.43 &    1814.47 &         12 \\
\hline
\hline
\multirow{13}{*}{\begin{sideways}SMALL ANGLE (SAD)\end{sideways}} 
&     3lmbd &       0.03 &       0.03 &       1.63 &          3 &       0.03 &       0.03 &       1.19 &          3 &       0.03 &       0.03 &       1.29 &          1 \\
\cline{2-14}
&    29edin &       0.85 &       0.84 &    1800.96 &        302 &       0.70 &       0.70 &    1819.15 &        102 &       0.69 &       0.67 &    1837.01 &         52 \\
\cline{2-14}
 &     30as &       0.08 &       0.08 &      14.71 &          5 &            &            &            &            &            &            &            &            \\
\cline{2-14}
&     30fsr &       0.08 &       0.08 &      18.55 &          9 &            &            &            &            &            &            &            &            \\
\cline{2-14}
&   118ieee &       4.02 &       3.98 &    1801.68 &        216 &       3.07 &       2.43 &    1811.48 &         71 &       3.35 &       3.07 &    1804.74 &         26 \\
\cline{2-14}
&   162ieee &       4.12 &       3.84 &    1813.33 &        118 &       3.76 &       3.76 &    1820.50 &         15 &            &            &            &            \\
\cline{2-14}
&   189edin* &       4.04 &       4.04 &     185.41 &         34 &       1.14 &       1.14 &    1800.41 &        139 &       2.70 &       1.06 &    1814.79 &         96 \\
\cline{2-14}
 &  300ieee &       0.21 &       0.16 &    1805.98 &        261 &       0.10 &       0.10 &     797.41 &          3 &            &            &            &            \\
\cline{2-14}
&    2383wp &       3.56 &       3.56 &    1813.15 &         66 &       3.03 &       2.89 &    1867.73 &         18 &       3.08 &       2.83 &    1805.81 &         12 \\
\cline{2-14}
&    2736sp &       2.77 &       2.77 &    1810.32 &         62 &       2.98 &       2.50 &    1827.66 &         31 &       3.81 &       3.81 &    1879.16 &         16 \\
\cline{2-14}
&   2737sop &       6.82 &       6.82 &    1808.89 &         71 &       3.57 &       3.57 &    1821.21 &         22 &       2.81 &       2.81 &    1805.93 &         15 \\
\cline{2-14}
&   2746wop &       6.36 &       5.61 &    1836.42 &         53 &       4.56 &       4.56 &    1899.28 &         26 &       5.41 &       5.41 &    1850.18 &         19 \\
\cline{2-14}
&    2746wp &       3.50 &       3.50 &    1806.90 &         53 &       2.74 &       2.74 &    1868.46 &         22 &       4.49 &       4.49 &    1860.37 &         15 \\
\hline

\end{tabular}  
\end{center}
\caption{Branch-and-cut results for the instances which are not solved at the root node up to the optimality threshold $\epsilon = 10^{-3}$. Percentage optimality gaps are reported after 15 and 30 minutes of branching (root node computation time is included).  OC: Operating Condition.}
\label{table:branch results}
\end{table}

\begin{table}\scriptsize
\begin{center}
\begin{tabular}{|c|r|r|r|r|r|r|r|r|r|r|r|r|}
\hline
  &      \multicolumn{ 4}{|c|}{$\mathsf{SEP(M)}$} &                    \multicolumn{ 4}{|c|}{$\mathsf{SEP(S)}$} &                \multicolumn{ 4}{|c|}{$\mathsf{SEP(M, S)}$} \\
\hline
          OC &      15 m &      30 m & time (s) &       node &      15 m &      30 m &  time (s)  &       node &      15 m &      30 m &  time (s)  &       node \\
\hline
       TYP &       0.71 &       0.68 &     730.44 &      70.24 &       0.43 &       0.40 &     562.17 &      17.00 &       0.48 &       0.48 &     589.65 &      13.57 \\
\hline
       API &       1.18 &       0.99 &     819.67 &      54.24 &       0.65 &       0.55 &     766.53 &      27.29 &       0.68 &       0.64 &     815.43 &      15.52 \\
\hline
       SAD &       1.75 &       1.70 &     787.31 &      60.05 &       1.24 &       1.18 &     829.07 &      22.00 &       1.46 &       1.36 &     852.02 &      12.57 \\
\hline
       ALL &       1.21 &       1.12 &     779.14 &      61.51 &       0.78 &       0.71 &     719.26 &      22.10 &       0.87 &       0.82 &     752.37 &      13.89 \\
\hline
\end{tabular}  
\end{center}
\caption{Averages of branch-and-cut results for all the instances. Percentage optimality gaps are reported after 15 and 30 minutes of branching (root node computation time is included).  }
\label{table:branch results average}
\end{table}



%
%

%
%
%
%
%
%
%
%
%
%
%

A final comparison of different methods is presented in Figure \ref{figure:CDF}, which can be interpreted as a \textit{cumulative distribution function}. In this figure, we record the fraction of instances solved up to a given percentage optimality gap. Therefore, a method whose corresponding curve is below the others is dominated. By construction, the plain SOCP relaxation approach is dominated since it is the weakest relaxation considered. Strong SOCP relaxations from our previous paper  \cite{kocuk2015} improves the plain SOCP relaxation considerably but it is not very competitive against the SDP relaxation, especially for easier instances. However, the proposed $\mathsf{SEP(S)}$ approach at the root node and after branching is  more successful than the SDP relaxation, especially for the more difficult instances.

\begin{figure}[H]
\centering
\begin{tikzpicture} 
\begin{axis}[ height=10cm, width=10cm, legend pos=south east, xlabel=\% Optimality Gap, ylabel=Fraction of instances, 
xmin=0, xmax=10, ymin=0, ymax=1, extra y ticks={0.1,0.3,0.5,0.7,0.9}, extra x ticks={1,3,5,7,9} ]  

\addlegendentry{{\small $\mathsf{SEP(S)}$-30 min}} \addplot [color=blue,mark= ]
coordinates 
{ 
(0.1,0.619047619047619)
(0.2,0.634920634920635)
(0.3,0.650793650793651)
(0.4,0.666666666666667)
(0.5,0.682539682539683)
(0.6,0.682539682539683)
(0.7,0.698412698412698)
(0.8,0.714285714285714)
(0.9,0.746031746031746)
(1,0.746031746031746)
(1.1,0.793650793650794)
(1.2,0.825396825396825)
(1.3,0.841269841269841)
(1.4,0.857142857142857)
(1.5,0.857142857142857)
(1.6,0.857142857142857)
(1.7,0.857142857142857)
(1.8,0.857142857142857)
(1.9,0.857142857142857)
(2,0.857142857142857)
(2.1,0.857142857142857)
(2.2,0.873015873015873)
(2.3,0.873015873015873)
(2.4,0.873015873015873)
(2.5,0.904761904761905)
(2.6,0.904761904761905)
(2.7,0.904761904761905)
(2.8,0.920634920634921)
(2.9,0.936507936507937)
(3,0.936507936507937)
(3.1,0.936507936507937)
(3.2,0.936507936507937)
(3.3,0.936507936507937)
(3.4,0.936507936507937)
(3.5,0.936507936507937)
(3.6,0.952380952380952)
(3.7,0.952380952380952)
(3.8,0.968253968253968)
(3.9,0.968253968253968)
(4,0.968253968253968)
(4.1,0.968253968253968)
(4.2,0.968253968253968)
(4.3,0.968253968253968)
(4.4,0.968253968253968)
(4.5,0.968253968253968)
(4.6,0.984126984126984)
(4.7,0.984126984126984)
(4.8,0.984126984126984)
(4.9,0.984126984126984)
(5,0.984126984126984)
(5.1,0.984126984126984)
(5.2,0.984126984126984)
(5.3,0.984126984126984)
(5.4,0.984126984126984)
(5.5,0.984126984126984)
(5.6,0.984126984126984)
(5.7,0.984126984126984)
(5.8,0.984126984126984)
(5.9,0.984126984126984)
(6,0.984126984126984)
(6.1,0.984126984126984)
(6.2,1)
(6.3,1)
(6.4,1)
(6.5,1)
(6.6,1)
(6.7,1)
(6.8,1)
(6.9,1)
(7,1)
(7.1,1)
(7.2,1)
(7.3,1)
(7.4,1)
(7.5,1)
(7.6,1)
(7.7,1)
(7.8,1)
(7.9,1)
(8,1)
(8.1,1)
(8.2,1)
(8.3,1)
(8.4,1)
(8.5,1)
(8.6,1)
(8.7,1)
(8.8,1)
(8.9,1)
(9,1)
(9.1,1)
(9.2,1)
(9.3,1)
(9.4,1)
(9.5,1)
(9.6,1)
(9.7,1)
(9.8,1)
(9.9,1)
(10,1)
(10.1,1)
 };

\addlegendentry{{\small $\mathsf{SEP(S)}$-Root}} \addplot [color=cyan,mark= ]
coordinates 
{ 
(0.1,0.53968253968254)
(0.2,0.571428571428571)
(0.3,0.603174603174603)
(0.4,0.619047619047619)
(0.5,0.634920634920635)
(0.6,0.634920634920635)
(0.7,0.650793650793651)
(0.8,0.682539682539683)
(0.9,0.698412698412698)
(1,0.698412698412698)
(1.1,0.761904761904762)
(1.2,0.777777777777778)
(1.3,0.777777777777778)
(1.4,0.80952380952381)
(1.5,0.80952380952381)
(1.6,0.825396825396825)
(1.7,0.825396825396825)
(1.8,0.825396825396825)
(1.9,0.825396825396825)
(2,0.825396825396825)
(2.1,0.825396825396825)
(2.2,0.825396825396825)
(2.3,0.825396825396825)
(2.4,0.825396825396825)
(2.5,0.825396825396825)
(2.6,0.841269841269841)
(2.7,0.841269841269841)
(2.8,0.841269841269841)
(2.9,0.841269841269841)
(3,0.841269841269841)
(3.1,0.873015873015873)
(3.2,0.873015873015873)
(3.3,0.873015873015873)
(3.4,0.888888888888889)
(3.5,0.888888888888889)
(3.6,0.888888888888889)
(3.7,0.904761904761905)
(3.8,0.920634920634921)
(3.9,0.920634920634921)
(4,0.936507936507937)
(4.1,0.936507936507937)
(4.2,0.936507936507937)
(4.3,0.952380952380952)
(4.4,0.952380952380952)
(4.5,0.952380952380952)
(4.6,0.968253968253968)
(4.7,0.968253968253968)
(4.8,0.968253968253968)
(4.9,0.968253968253968)
(5,0.968253968253968)
(5.1,0.968253968253968)
(5.2,0.984126984126984)
(5.3,0.984126984126984)
(5.4,0.984126984126984)
(5.5,0.984126984126984)
(5.6,0.984126984126984)
(5.7,0.984126984126984)
(5.8,0.984126984126984)
(5.9,0.984126984126984)
(6,0.984126984126984)
(6.1,0.984126984126984)
(6.2,0.984126984126984)
(6.3,0.984126984126984)
(6.4,0.984126984126984)
(6.5,0.984126984126984)
(6.6,0.984126984126984)
(6.7,0.984126984126984)
(6.8,0.984126984126984)
(6.9,0.984126984126984)
(7,0.984126984126984)
(7.1,0.984126984126984)
(7.2,0.984126984126984)
(7.3,0.984126984126984)
(7.4,0.984126984126984)
(7.5,0.984126984126984)
(7.6,0.984126984126984)
(7.7,0.984126984126984)
(7.8,0.984126984126984)
(7.9,0.984126984126984)
(8,1)
(8.1,1)
(8.2,1)
(8.3,1)
(8.4,1)
(8.5,1)
(8.6,1)
(8.7,1)
(8.8,1)
(8.9,1)
(9,1)
(9.1,1)
(9.2,1)
(9.3,1)
(9.4,1)
(9.5,1)
(9.6,1)
(9.7,1)
(9.8,1)
(9.9,1)
(10,1)
(10.1,1)
 };

\addlegendentry{{\small SDP}} \addplot [color=red,mark= ]
coordinates 
{ 
(0.1,0.634920634920635)
(0.2,0.666666666666667)
(0.3,0.666666666666667)
(0.4,0.698412698412698)
(0.5,0.761904761904762)
(0.6,0.761904761904762)
(0.7,0.777777777777778)
(0.8,0.777777777777778)
(0.9,0.777777777777778)
(1,0.777777777777778)
(1.1,0.80952380952381)
(1.2,0.825396825396825)
(1.3,0.888888888888889)
(1.4,0.888888888888889)
(1.5,0.888888888888889)
(1.6,0.888888888888889)
(1.7,0.888888888888889)
(1.8,0.888888888888889)
(1.9,0.888888888888889)
(2,0.888888888888889)
(2.1,0.904761904761905)
(2.2,0.904761904761905)
(2.3,0.904761904761905)
(2.4,0.904761904761905)
(2.5,0.904761904761905)
(2.6,0.904761904761905)
(2.7,0.904761904761905)
(2.8,0.904761904761905)
(2.9,0.904761904761905)
(3,0.904761904761905)
(3.1,0.904761904761905)
(3.2,0.904761904761905)
(3.3,0.904761904761905)
(3.4,0.904761904761905)
(3.5,0.904761904761905)
(3.6,0.920634920634921)
(3.7,0.920634920634921)
(3.8,0.920634920634921)
(3.9,0.920634920634921)
(4,0.920634920634921)
(4.1,0.920634920634921)
(4.2,0.920634920634921)
(4.3,0.920634920634921)
(4.4,0.920634920634921)
(4.5,0.920634920634921)
(4.6,0.920634920634921)
(4.7,0.920634920634921)
(4.8,0.920634920634921)
(4.9,0.920634920634921)
(5,0.920634920634921)
(5.1,0.920634920634921)
(5.2,0.920634920634921)
(5.3,0.936507936507937)
(5.4,0.936507936507937)
(5.5,0.936507936507937)
(5.6,0.936507936507937)
(5.7,0.936507936507937)
(5.8,0.936507936507937)
(5.9,0.936507936507937)
(6,0.936507936507937)
(6.1,0.936507936507937)
(6.2,0.936507936507937)
(6.3,0.936507936507937)
(6.4,0.936507936507937)
(6.5,0.936507936507937)
(6.6,0.936507936507937)
(6.7,0.936507936507937)
(6.8,0.936507936507937)
(6.9,0.936507936507937)
(7,0.936507936507937)
(7.1,0.936507936507937)
(7.2,0.936507936507937)
(7.3,0.936507936507937)
(7.4,0.936507936507937)
(7.5,0.936507936507937)
(7.6,0.952380952380952)
(7.7,0.952380952380952)
(7.8,0.952380952380952)
(7.9,0.952380952380952)
(8,0.952380952380952)
(8.1,0.952380952380952)
(8.2,0.952380952380952)
(8.3,0.952380952380952)
(8.4,0.952380952380952)
(8.5,0.952380952380952)
(8.6,0.952380952380952)
(8.7,0.952380952380952)
(8.8,0.952380952380952)
(8.9,0.952380952380952)
(9,0.952380952380952)
(9.1,0.952380952380952)
(9.2,0.952380952380952)
(9.3,0.952380952380952)
(9.4,0.952380952380952)
(9.5,0.952380952380952)
(9.6,0.952380952380952)
(9.7,0.952380952380952)
(9.8,0.952380952380952)
(9.9,0.952380952380952)
(10,0.952380952380952)
(10.1,0.952380952380952)
 };

\addlegendentry{{\small Best of \cite{kocuk2015}}} \addplot [color=green,mark= ]
coordinates 
{ 
(0.1,0.428571428571429)
(0.2,0.46031746031746)
(0.3,0.492063492063492)
(0.4,0.53968253968254)
(0.5,0.571428571428571)
(0.6,0.571428571428571)
(0.7,0.571428571428571)
(0.8,0.603174603174603)
(0.9,0.603174603174603)
(1,0.603174603174603)
(1.1,0.603174603174603)
(1.2,0.603174603174603)
(1.3,0.619047619047619)
(1.4,0.634920634920635)
(1.5,0.666666666666667)
(1.6,0.73015873015873)
(1.7,0.73015873015873)
(1.8,0.746031746031746)
(1.9,0.746031746031746)
(2,0.746031746031746)
(2.1,0.777777777777778)
(2.2,0.777777777777778)
(2.3,0.777777777777778)
(2.4,0.793650793650794)
(2.5,0.80952380952381)
(2.6,0.80952380952381)
(2.7,0.80952380952381)
(2.8,0.825396825396825)
(2.9,0.825396825396825)
(3,0.825396825396825)
(3.1,0.825396825396825)
(3.2,0.825396825396825)
(3.3,0.825396825396825)
(3.4,0.825396825396825)
(3.5,0.825396825396825)
(3.6,0.857142857142857)
(3.7,0.873015873015873)
(3.8,0.873015873015873)
(3.9,0.888888888888889)
(4,0.904761904761905)
(4.1,0.904761904761905)
(4.2,0.904761904761905)
(4.3,0.904761904761905)
(4.4,0.904761904761905)
(4.5,0.904761904761905)
(4.6,0.904761904761905)
(4.7,0.904761904761905)
(4.8,0.904761904761905)
(4.9,0.904761904761905)
(5,0.904761904761905)
(5.1,0.904761904761905)
(5.2,0.904761904761905)
(5.3,0.904761904761905)
(5.4,0.904761904761905)
(5.5,0.904761904761905)
(5.6,0.904761904761905)
(5.7,0.904761904761905)
(5.8,0.904761904761905)
(5.9,0.920634920634921)
(6,0.920634920634921)
(6.1,0.920634920634921)
(6.2,0.920634920634921)
(6.3,0.936507936507937)
(6.4,0.936507936507937)
(6.5,0.936507936507937)
(6.6,0.936507936507937)
(6.7,0.936507936507937)
(6.8,0.936507936507937)
(6.9,0.936507936507937)
(7,0.936507936507937)
(7.1,0.936507936507937)
(7.2,0.936507936507937)
(7.3,0.936507936507937)
(7.4,0.936507936507937)
(7.5,0.952380952380952)
(7.6,0.952380952380952)
(7.7,0.952380952380952)
(7.8,0.952380952380952)
(7.9,0.952380952380952)
(8,0.952380952380952)
(8.1,0.952380952380952)
(8.2,0.952380952380952)
(8.3,0.952380952380952)
(8.4,0.952380952380952)
(8.5,0.952380952380952)
(8.6,0.952380952380952)
(8.7,0.952380952380952)
(8.8,0.952380952380952)
(8.9,0.952380952380952)
(9,0.952380952380952)
(9.1,0.952380952380952)
(9.2,0.952380952380952)
(9.3,0.952380952380952)
(9.4,0.952380952380952)
(9.5,0.952380952380952)
(9.6,0.952380952380952)
(9.7,0.952380952380952)
(9.8,0.952380952380952)
(9.9,0.952380952380952)
(10,0.952380952380952)
(10.1,0.952380952380952)
 };

\addlegendentry{{\small SOCP}} \addplot [color=magenta,mark= ]
coordinates 
{ (0.1,0.111111111111111)
(0.2,0.174603174603175)
(0.3,0.206349206349206)
(0.4,0.222222222222222)
(0.5,0.253968253968254)
(0.6,0.285714285714286)
(0.7,0.333333333333333)
(0.8,0.333333333333333)
(0.9,0.380952380952381)
(1,0.396825396825397)
(1.1,0.412698412698413)
(1.2,0.428571428571429)
(1.3,0.444444444444444)
(1.4,0.476190476190476)
(1.5,0.476190476190476)
(1.6,0.523809523809524)
(1.7,0.53968253968254)
(1.8,0.53968253968254)
(1.9,0.53968253968254)
(2,0.53968253968254)
(2.1,0.555555555555556)
(2.2,0.571428571428571)
(2.3,0.571428571428571)
(2.4,0.587301587301587)
(2.5,0.603174603174603)
(2.6,0.603174603174603)
(2.7,0.603174603174603)
(2.8,0.603174603174603)
(2.9,0.603174603174603)
(3,0.619047619047619)
(3.1,0.619047619047619)
(3.2,0.619047619047619)
(3.3,0.619047619047619)
(3.4,0.634920634920635)
(3.5,0.650793650793651)
(3.6,0.650793650793651)
(3.7,0.682539682539683)
(3.8,0.698412698412698)
(3.9,0.698412698412698)
(4,0.698412698412698)
(4.1,0.698412698412698)
(4.2,0.714285714285714)
(4.3,0.73015873015873)
(4.4,0.746031746031746)
(4.5,0.746031746031746)
(4.6,0.746031746031746)
(4.7,0.746031746031746)
(4.8,0.761904761904762)
(4.9,0.777777777777778)
(5,0.777777777777778)
(5.1,0.777777777777778)
(5.2,0.777777777777778)
(5.3,0.777777777777778)
(5.4,0.777777777777778)
(5.5,0.793650793650794)
(5.6,0.793650793650794)
(5.7,0.793650793650794)
(5.8,0.793650793650794)
(5.9,0.825396825396825)
(6,0.825396825396825)
(6.1,0.825396825396825)
(6.2,0.825396825396825)
(6.3,0.825396825396825)
(6.4,0.825396825396825)
(6.5,0.825396825396825)
(6.6,0.841269841269841)
(6.7,0.841269841269841)
(6.8,0.841269841269841)
(6.9,0.841269841269841)
(7,0.841269841269841)
(7.1,0.857142857142857)
(7.2,0.857142857142857)
(7.3,0.857142857142857)
(7.4,0.857142857142857)
(7.5,0.857142857142857)
(7.6,0.857142857142857)
(7.7,0.857142857142857)
(7.8,0.857142857142857)
(7.9,0.857142857142857)
(8,0.857142857142857)
(8.1,0.857142857142857)
(8.2,0.857142857142857)
(8.3,0.857142857142857)
(8.4,0.857142857142857)
(8.5,0.857142857142857)
(8.6,0.857142857142857)
(8.7,0.857142857142857)
(8.8,0.857142857142857)
(8.9,0.857142857142857)
(9,0.857142857142857)
(9.1,0.857142857142857)
(9.2,0.873015873015873)
(9.3,0.873015873015873)
(9.4,0.873015873015873)
(9.5,0.873015873015873)
(9.6,0.873015873015873)
(9.7,0.873015873015873)
(9.8,0.873015873015873)
(9.9,0.873015873015873)
(10,0.873015873015873)
(10.1,0.873015873015873)
 };

\end{axis} \end{tikzpicture}
\caption{Fraction of instances solved up to a given percentage optimality gap for different methods.}\label{figure:CDF}
\end{figure}
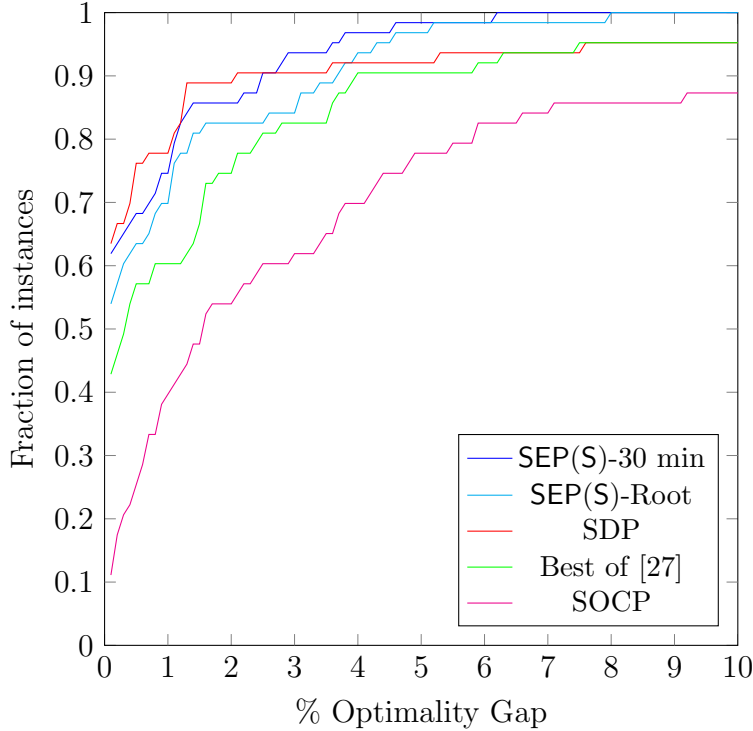

\section{Conclusion}
\label{sec: concl}

In this paper, we proposed new approaches for obtaining globally optimal solutions of the AC OPF Problem. We first reformulated the AC OPF problem as a minor constrained problem, and then proposed several convexification techniques for nonconvexities involving minor constraints using only second-order conic and linear relaxations.  We improved the resulting SOCP relaxation via cutting planes and convex envelopes by incorporating bound tightening techniques.  We proposed three methods with respect to the cutting plane procedure. Our methods are successful in proving global optimal solutions for many challenging OPF instances from the NESTA archive. Compared to the standard SDP relaxation, our approaches provide about 2 times smaller optimality gaps with only half of the average computation time. For the instances not solved, we propose to use a branch-and-cut scheme where the proposed SOCP relaxation serves as the root node relaxation. The strongest of our SOCP based branch-and-cut algorithms proves 0.71\% optimality gap in 720 seconds on the average for the NESTA library.

As a future work, we would like to apply the methodology developed in this chapter to multi-period AC OPF and Optimal Transmission Switching Problems. Another possible line of research is to implement an LP based outer-approximation to SOCP based branch-and-cut method. This would lead to weaker relaxations than a pure SOCP based approach but can incorporate warm-start and it may be possible to process significantly more nodes in the same amount of time.

\appendix
\section{KKT Points for Arctangent Envelopes}
\label{app:arctan}
Let us rewrite the optimization problem in \eqref{error problem with angle} as
\begin{equation}\label{error prob generic}
\Delta \gamma = \max \{  f(c,s) :   (c, s) \in F \},
\end{equation}
where 
$$
 f(c,s):= \arctan \left( \frac{s}{c} \right) - (\gamma + \alpha c + \beta s )  
$$
and
$$
 F := \{(c,s) :   (c, s) \in [\underline c, \overline c] \times [\underline s, \overline s],  \ c \tan \underline \theta \le s \le \tan \overline \theta \}.
$$
Without loss of generality, let us assume that the relations
$
\arctan(\underline s/ \underline c) \le \underline \theta
$
and
$
\overline \theta \le \arctan(\overline s/ \underline c)
$
hold between the variable bounds (otherwise, at least one of the bounds can be improved). 
Let us denote the optimal solution to problem \eqref{error prob generic} as $(c^*, s^*)$. First, we claim that $(c^*, s^*)$ is not in the interior of $F$. This is due to the fact that  the Hessian of $f$ at a point $(c,s)$, which is  given as
$$
\frac{1}{(c^2+s^2)^2}
\begin{bmatrix}
2cs & s^2-c^2\\ s^2-c^2 & -2cs
\end{bmatrix},
$$
is an indefinite matrix. Therefore, an interior point of $F$ will not satisfy the second-order necessary conditions of local optimality. 

The above argument implies that  $(c^*, s^*)$ belongs to the boundary of $F$, which is the union of the following six line segments:
\begin{enumerate}
\item $[(\underline c, \underline c \tan \underline \theta), (\underline c, \underline c \tan \overline \theta)]$
\item $[(\underline c, \underline c \tan \overline \theta), (\overline s/\tan \overline\theta, \overline s)]$ 
\item $[ (\overline s/\tan \overline\theta, \overline s), (\overline c, \overline s)]$ 
\item $[  (\overline c, \overline s), (\overline c, \underline s)]$ 
\item $[  (\overline c, \underline s),  (\underline s/\tan \underline\theta, \underline s)]$ 
\item $[ (\underline s/\tan \underline\theta, \underline s), (\underline c, \underline c \tan \underline \theta)]$
\end{enumerate}
Note that the line segments (ii) and (vi) cannot contain $(c^*, s^*)$ in their relative interior since the function $f$ is linear along them. Hence, the problem reduces to four 1-dimensional optimization problems, which can be solved easily. The global optimal solution $(c^*, s^*)$ is the one that gives the largest objective value among the KKT points calculated by solving those four 1-dimensional optimization problems.

\section{Algorithms for Extreme Point Calculations for Edge Cuts}
\label{app:extrpoint}

\begin{algorithm}[H]
\caption{Find extreme points of conv$(\mathcal{K}_{ij}^\ge)$ when $c_{jj}=\hat c_{jj}$, $c_{ij} = \hat c_{ij}$ and $s_{ij} = \hat s_{ij}$ are fixed to one of their bounds.}
\label{alg:extr cii}
\begin{algorithmic}
\STATE Compute $\phi = \frac{\hat c_{ij}^2 + \hat s_{ij}^2}{\hat c_{jj}}$ .
\STATE Let $E$ be the set of $c_{ii}$ coordinates of the extreme points.
\STATE $E= \begin{cases} 
\emptyset & \text{if } \phi < \underline c_{ii}\\ 
\{\phi, \overline c_{ii} \} & \text{if } \underline c_{ii} \le \phi \le \overline c_{ii} \\
\{\underline c_{ii}, \overline c_{ii} \} & \text{if }  \phi > \overline c_{ii} 
\end{cases}$
\end{algorithmic}
\end{algorithm}

\begin{algorithm}[H]
\caption{Find extreme points of conv$(\mathcal{K}_{ij}^\ge)$ when $c_{ii}=\hat c_{ii}$,  $c_{jj}=\hat c_{jj}$ and $s_{ij} = \hat s_{ij}$ are fixed to one of their bounds.}
\label{alg:extr cij}
\begin{algorithmic}
\STATE Compute $\phi =\hat c_{ii} \hat c_{jj} - \hat s_{ij}^2$.
\STATE Let $E$ be the set of  $c_{ij}$ coordinates of the extreme points.
\STATE $E= \begin{cases} 
\{\underline c_{ij}, \overline c_{ij} \}  & \text{if } \phi < 0   \\ 
\{\underline c_{ij}, \overline c_{ij} \}  & \text{if } \phi \ge 0  \text{ and } \sqrt{\phi} < \underline c_{ij}  \\ 
\{\sqrt{\phi}, \overline c_{ij} \}  & \text{if } \phi \ge 0  \text{ and }  \underline c_{ij} \le \sqrt{\phi} \le \overline c_{ij}  \\ 
\emptyset &   \text{if } \phi \ge 0  \text{ and }   \sqrt{\phi} >  \overline c_{ij} 
\end{cases}$
\end{algorithmic}
\end{algorithm}

\begin{algorithm}[H]
\caption{Find extreme points of conv$(\mathcal{K}_{ij}^\ge)$ when $c_{ii}=\hat c_{ii}$,  $c_{jj}=\hat c_{jj}$ and $c_{ij} = \hat c_{ij}$ are fixed to one of their bounds.}
\label{alg:extr sij}
\begin{algorithmic}
\STATE Compute $\phi =\hat c_{ii} \hat c_{jj} - \hat c_{ij}^2$.
\STATE Let $E$ be the set of  $s_{ij}$ coordinates of the extreme points.
\STATE $E= \begin{cases} 
\{\underline s_{ij}, \overline s_{ij} \}  & \text{if } \phi < 0   \\ 
\{\underline s_{ij}, \overline s_{ij} \}  & \text{if } \phi \ge 0  \text{ and } \overline s_{ij} < -\sqrt{\phi}   \\ 
\{ \underline s_{ij}, -\sqrt{\phi}\}  & \text{if } \phi \ge 0  \text{ and }  \underline s_{ij} \le -\sqrt{\phi} \le \overline s_{ij}  \\ 
\{\underline s_{ij}, -\sqrt{\phi}, \sqrt{\phi}, \overline s_{ij} \} &   \text{if } \phi \ge 0  \text{ and }  \underline s_{ij} \le -  \sqrt{\phi}, \sqrt{\phi} \ge \overline s_{ij} \\ 
\emptyset &   \text{if } \phi \ge 0  \text{ and } -\sqrt{\phi} \le \underline s_{ij}   , \overline s_{ij} \le \sqrt{\phi}  \\
\{ \sqrt{\phi}, \overline s_{ij}\} &   \text{if } \phi \ge 0  \text{ and } -\sqrt{\phi} \le \underline s_{ij}  \le \sqrt{\phi}  \le \overline s_{ij} \\
\{\underline s_{ij}, \overline s_{ij} \}  & \text{if } \phi \ge 0  \text{ and }    \underline s_{ij} > \sqrt{\phi}
\end{cases}$
\end{algorithmic}
\end{algorithm}

\section{Proof of Theorem \ref{main theorem 1}}
\label{app:main theorem 1 proof}

Our proof approach is based on identifying the extreme points of $\mathcal{S}_a$.
Let us start with a proposition.

\begin{prop} \label{either or 1}
Let $(x,y)$ be an extreme point of the set $\mathcal{S}_a$.  Then, for a distinct pair of indices  $i$ and $j$ 
\begin{enumerate}
\item either $x_i$ or $y_j$ is at one of its bounds.
\item either $x_i$ or $y_i$ is at one of its bounds.
\item either $x_j$ or $y_j$ is at one of its bounds.
\item either $x_j$ or $y_i$ is at one of its bounds.
\end{enumerate}
\end{prop}
\begin{proof}
We only prove the first statement. The others can be proven using exactly the same reasoning.

Assume for a contradiction that $\underline x_i < x_i < \overline x_i$ and $\underline y_j < y_j < \overline y_j$. Consider the following cases:
\begin{enumerate}[{Case} 1:]
\item $y_i \neq 0$ and $x_j \neq0$
\begin{enumerate}[{Case 1}a:]
\item $\frac{a_i y_i}{a_jx_j} > 0$

Let $\epsilon = \{x_i - \underline x_i, \overline x_i - x_i, \frac{a_jx_j}{a_i y_i} (y_i - \underline y_i), \frac{a_jx_j}{a_i y_i} (\overline y_i -  y_i) \} $ and $\delta = \frac{a_i y_i}{a_jx_j} \epsilon$. Note that both $\epsilon$ and $\delta$ are positive.
Now, construct $(x^+, y^-) = (x+\epsilon e_i, y-\delta e_j)$ and $(x^-, y^+) = (x-\epsilon e_i, y+\delta e_j)$ where $e_i$ is the $i$-th unit vector. Observe that both $(x^+, y^-)$ and $(x^-, y^+)$ belong to $\mathcal{S}_a$. Moreover, $(x,y) = \frac12 (x^+, y^-) + \frac12 (x^-, y^+)$. But, this is a contradiction to $(x,y)$ being an extreme point of~$\mathcal{S}_a$.

\item $\frac{a_i y_i}{a_jx_j} < 0$

Let $\epsilon = \{x_i - \underline x_i, \overline x_i - x_i, \frac{a_jx_j}{a_i y_i} (\underline y_i -  y_i), \frac{a_jx_j}{a_i y_i} ( y_i -\overline  y_i) \} $ and $\delta = -\frac{a_i y_i}{a_jx_j} \epsilon$. Note that both $\epsilon$ and $\delta$ are positive.
Now, construct $(x^+, y^+) = (x+\epsilon e_i, y+\delta e_j)$ and $(x^-, y^-) = (x-\epsilon e_i, y-\delta e_j)$. Observe that both $(x^+, y^+)$ and $(x^-, y^-)$ belong to $\mathcal{S}_a$. Moreover, $(x,y) = \frac12 (x^+, y^+) + \frac12 (x^-, y^-)$. But, this is a contradiction to $(x,y)$ being an extreme point of $\mathcal{S}_a$.
\end{enumerate}

\item $y_i = 0$

Let $\epsilon = \{x_i - \underline x_i, \overline x_i - x_i \} $. Note that  $\epsilon$ is positive.
Now, construct $(x^+, y) = (x+\epsilon e_i, y)$ and $(x^-, y) = (x-\epsilon e_i, y)$. Observe that both $(x^+, y)$ and $(x^-, y)$ belong to $\mathcal{S}_a$. Moreover, $(x,y) = \frac12 (x^+, y) + \frac12 (x^-, y)$. But, this is a contradiction to $(x,y)$ being an extreme point of $\mathcal{S}_a$.

\item $x_j = 0$

Let $\delta = \{y_i - \underline y_i, \overline y_i - y_i \} $. Note that  $\delta$ is positive.
Now, construct $(x, y^+) = (x, y+\delta e_j)$ and $(x, y^-) = (x, y-\epsilon e_j)$. Observe that both $(x, y^+)$ and $(x, y^-)$ belong to $\mathcal{S}_a$. Moreover, $(x,y) = \frac12 (x, y^+) + \frac12 (x, y^-)$. But, this is a contradiction to $(x,y)$ being an extreme point of $\mathcal{S}_a$.
\end{enumerate}
\end{proof}

Proposition \ref{either or 1} implies the following corollary.
\begin{cor} \label{and cor}
Let $(x,y)$ be an extreme point of the set $\mathcal{S}_a$.  Then, either $x_i$ and $y_i$ or  $x_j$ and $y_j$ are at their bounds  for a distinct pair of indices  $i$ and $j$.
\end{cor}
\begin{proof}
Let $x_i = \hat x_i$ be a shorthand for ``either $x_i = \underline x_i $ or $x_i = \overline x_i $". Then, Proposition \ref{either or 1} implies that 
\begin{align}
\begin{split}
& (x_i = \hat x_i \lor y_j = \hat y_j) \land (x_i = \hat x_i \lor x_j = \hat x_j) \land (y_i = \hat y_i \lor y_j = \hat y_j) \land (y_i = \hat y_i \lor x_j = \hat x_j)  \\ = &  (x_i = \hat x_i \land y_i = \hat y_i) \lor (x_j = \hat x_j \land y_j = \hat y_j),
\end{split}
\end{align}
which is the desired conclusion.
\end{proof}

%

An immediate consequence of Corollary \ref{and cor} is the following characterization of extreme points of $\mathcal{S}_a$: 
\begin{cor} 
All the extreme points of $\mathcal{S}_a$ are in one of the following sets:
\begin{itemize}
\item
$D_0 = \{(x,y) \in \mathcal{S}_a: (x_i,y_i) = (\hat x_i, \hat y_i)\ \forall i \}$
\item
$D_k = \{(x,y) \in \mathcal{S}_a: (x_i,y_i) = (\hat x_i, \hat y_i) \ i \neq k, x_ky_k = -\frac{1}{a_k} \sum_{i \neq k} a_i \hat x_i \hat y_i, \ x_k\in[\underline x_k, \overline x_k], \ y_k\in[\underline y_k , \overline y_k] \} \quad k=1,\dots,N$
\end{itemize}
\end{cor}

Note that $D_0$ is a collection of at most $4^N$ singletons whereas $D_k$ is a collection of $4^{N-1}$ sets for each $k$. The projection of such a set onto $(x_k,y_k)$ is of the following form
\begin{equation}
T_\alpha = \{(x,y) \in \mathbb{R}^2 : xy = \alpha, \  \ x\in[\underline x, \overline x], \ y\in[\underline y , \overline y] \}
\end{equation}
for some constant $\alpha$.

\begin{prop} 
Set conv$(T_\alpha)$ is second-order cone representable for any value of $\alpha$.
\end{prop}
There are several cases based on parameter values. In the most complicated case,  we need $xy \ge \alpha$ (which is conic representable) and McCormick envelopes.

Now, we are ready to prove the main result.
\begin{proof} [Proof of Theorem \ref{main theorem 1}]
Since the convex hull of all the  disjunctions are second-order cone representable (could be polyhedral or singleton depending on parameter values), conv$(\mathcal{S}_a)$ is also second-order cone representable.
\end{proof}

\bibliographystyle{plain}
\bibliography{references}

\end{document}